\documentclass[11pt]{amsart}
\usepackage{mathrsfs}
\usepackage{amsfonts}
\usepackage{amsthm}
\usepackage{amssymb}
\usepackage{latexsym}
\usepackage{tikz}
\usetikzlibrary{calc,shapes.arrows,decorations.pathreplacing, patterns}
\newtheorem{theorem}{Theorem}
\newtheorem{lemma}{Lemma}
\newtheorem{corollary}{Corollary}

\newtheorem{proposition}{Proposition}

\begin{document}

\title{$L^{2}$-Sobolev theory for the complex Green operator}

\author{S\'{e}verine Biard}
\author{Emil J.~Straube}

\address{Department of Mathematics Texas A\&M University College Station, Texas, 77843}
\email{biards@math.tamu.edu}
\email{straube@math.tamu.edu}

\thanks{2000 \emph{Mathematics Subject Classification}: 32W10, 32V20}
\keywords{Complex Green operator, $\overline{\partial}_{M}$, CR-submanifold of hypersurface type, closed range, compactness, Sobolev estimates}
\thanks{Supported in part by Qatar National Research Fund Grant NPRP 7-511-1-98 .}

\date{June 1, 2016; revised May 1, 2017}

\begin{abstract}
These notes are concerned with the $L^{2}$-Sobolev theory of the complex Green operator on pseudoconvex, oriented, bounded and closed CR-submanifolds of $\mathbb{C}^{n}$ of hypersurface type. This class of submanifolds generalizes that of boundaries of pseudoconvex domains. We first discuss briefly the CR-geometry of general CR-submanifolds and then specialize to this class. Next, we review the basic $L^{2}$-theory of the tangential Cauchy--Riemann operator and the associated complex Green operator(s) on these submanifolds. After these preparations, we discuss recent results on compactness and regularity in Sobolev spaces of the complex Green operator(s).
\end{abstract}

\maketitle

\section{Introduction}

This paper gives an overview of some recent developments concerning estimates (compactness, Sobolev) for the complex Green operator on CR-submanifolds of $\mathbb{C}^{n}$ of hypersurface type. These submanifolds are assumed compact, without boundary, and pseudoconvex; except in section \ref{Compact}, they are also assumed orientable. On the one hand, results in the $L^{2}$-Sobolev theory for the complex Green operator on this type of CR-submanifold by now more or less mirror those for the $\overline{\partial}$-Neumann operator on a pseudoconvex domain. On the other hand, many new issues arise, and interesting new ideas are required. 

Such a submanifold $M$ can locally be graphed over an actual hypersurface. In section \ref{CRsubmanifolds}, we show how the observation that the local graphing functions are CR-functions on that hypersurface combines with results on extendability of CR-functions to construct a one-sided complexification $\widehat{M}$ of $M$. Once $\widehat{M}$ is in place, one can `fill in the hole', at the expense of going outside $\mathbb{C}^{n}$, to obtain a complex manifold whose boundary is $M$. These results are used in section \ref{L-2} to obtain closed range in $L^{2}$ for $\overline{\partial}_{M}$, the Cauchy--Riemann operator on $M$. In this section, we then recall the basic facts about the $L^{2}$-theory and the complex Green operator. Concerning compactness, discussed in section \ref{Compact}, differences to the $\overline{\partial}$-theory on pseudoconvex domains emerge. Namely, compactness of the complex Green operator holds at symmetric form levels, and compactness does not necessarily pass to higher form levels; both are in contrast to what happens for the $\overline{\partial}$-Neumann operator. Once these general facts are 
explained, we introduce the familiar potential theoretic sufficient condition for compactness of the $\overline{\partial}$-Neumann operator, known as property ($P$), and show that when assumed at symmetric form levels, it is also a sufficient condition for compactness of the Green operator. The section concludes with a discussion of some geometric sufficient conditions for compactness. In this case the assumptions are independent of the form level, so that there is no need to `symmetrize' them. The final section, section \ref{Sobolev}, is dedicated to Sobolev estimates for the complex Green operator. The main result is that if a certain one-form on $M$ is exact on the null space of the Levi form, then Sobolev estimates do hold. This form is exact on the null space of the Levi form when $M$ is defined by a set of plurisubharmonic functions. It is also exact when $M$ is strictly pseudoconvex except for a simply connected complex submanifold. More generally, this form is closed on the null space of the 
Levi form, and so induces a DeRham cohomology class on a complex submanifold. It suffices that this class vanishes. In particular, a complex annulus in $M$ may or may not be bad for Sobolev estimates, while a disc is always benign.

Our survey originates with a set of notes, taken by the first author, of four lectures given by the second author at the January 2016 School in Complex Analysis and Geometry at the Tsinghua Sanya International Mathematics Forum, Sanya, China. What we present here is a substantially expanded version. The support of the Mathematics Forum is gratefully acknowledged.

\section{CR-submanifolds of hypersurface type}\label{CRsubmanifolds}

In this section, we give a brief introduction to the geometry of Cauchy--Riemann (CR-) submanifolds of $\mathbb{C}^n$, and to CR-functions and mappings. This introduction concentrates on what is needed to discuss the construction of a one-sided complexification of an oriented smooth bounded compact pseudoconvex CR-submanifold of $\mathbb{C}^{n}$ of hypersurface type, started in \cite{MunasingheStraube12} and completed in \cite{Baracco1}. For background on CR-submanifolds, we refer the reader to the books \cite{Boggess91, BER99, Zampieri08}.

\smallskip

Denote by $J$ the complex structure on $\mathbb{C}^n\simeq \mathbb{R}^{2n}$. Here, the identification is the standard one, i.e $(x_{1}+iy_{1}, \hdots, x_{n}+iy_{n}) \simeq (x_{1}, y_{1}, \hdots, x_{n}, y_{n})$. Then $J(z_{1},\hdots, z_{n}) = (iz_{1},\hdots, iz_{n})$; equivalently, $J(x_{1}, y_{1}, \hdots, x_{n}, y_{n}) = (-y_{1}, x_{1}, \hdots, -y_{n}, x_{n})$. 

Let $M$ be a smooth submanifold in $\mathbb{C}^n$.
We denote by $T_pM$ the real tangent space to $M$ at the point $p\in M$.
$T_pM$ is in general not invariant under the complex structure $J$; we define \emph{the complex tangent space to $M$ at $p\in M$} to be the largest subspace of $T_pM$ that is invariant under $J$. That is, we set
\begin{equation}\label{comptan}
T_p^{\mathbb{C}}M:=T_pM\cap JT_pM\;.
\end{equation}
Note that $T_{p}^{\mathbb{C}}M$ is naturally a complex subspace of $\mathbb{C}^{n}$; in particular, it is a vector space over $\mathbb{C}$ (compare also \eqref{complextangent} below).

A submanifold $M$ in $\mathbb{C}^n$ is called a \textit{CR-submanifold} of $\mathbb{C}^n$ if its CR-dimension, $\text{dim}_{\mathbb{C}}T_p^{\mathbb{C}}M$, is independent of $p\in M$. In this case, we define the \emph{CR-codimension of $M$} to be the codimension of $T^{\mathbb{C}}M$ inside $TM$ (which is also independent of the point).

Here are some examples of CR-submanifolds.

\begin{enumerate}
\item A totally real submanifold is a CR-submanifold: $T_p^{\mathbb{C}}M=\lbrace{0}\rbrace$ for every $p\in M$.
\item Any real hypersurface is a CR-submanifold: $\text{dim}_{\mathbb{C}}T_p^{\mathbb{C}}M = (n-1)$ for every $p\in M$.
\item A complex submanifold is also a CR-submanifold: $T_pM=T^{\mathbb{C}}_pM$ for every $p\in M$.
\item The boundary of a smooth complex submanifold $S\subset\mathbb{C}^{n}$: $T_{p}^{\mathbb{C}}bS = (m-1)$ for every point, where $m$ is the dimension of $S$.
\end{enumerate}

In a neighborhood $U$ of a point $p\in M$, a smooth CR-submanifold of $\mathbb{C}^n$ of real codimension $k\geq1$ can be given by a set of \textit{defining functions} as $M = 
\lbrace{z\in U\mid \rho_1(z)=\dots=\rho_{k}(z)=0}\rbrace$,
where $\lbrace{\rho_i}\rbrace_{1\leq i\leq k}$ are smooth real-valued functions in $U$ verifying $d\rho_1\wedge\dots\wedge d\rho_k\neq 0$ on $M$. Then we have
\begin{equation}\label{complextangent}
(\zeta_1,\dots, \zeta_n)\in T^{\mathbb{C}}_pM \Leftrightarrow \sum_{j=1}^n\dfrac{\partial\rho_i}{\partial z_j}(p)\zeta_j=0,\quad 1\leq i\leq k.
\end{equation}
Indeed, if $\zeta_{j} = \xi_{j}+i\eta_{j}$, $1\leq j\leq n$, then taking the real part of the right-hand side of \eqref{complextangent} says that $\zeta = (\zeta_{1}, \hdots, \zeta_{n}) = (\xi_{1}, \eta_{1}, \hdots, \xi_{n}, \eta_{n})$ belongs to $T_{p}M$, whereas taking the imaginary part says that $J\zeta = (-\eta_{1}, \xi_{1}, \hdots, -\eta_{n}, \xi_{n})$ belongs to $T_{p}M$.
It is important to note that in general, not all $k$ equations in \eqref{complextangent} will be linearly independent over $\mathbb{C}$ (although the real differentials $d\rho_{1}, \hdots, d\rho_{k}$ are linearly independent over $\mathbb{R}$). 

\smallskip

A vector field $X(z) = \sum_{j=1}^{n}a_{j}(z)\partial/\partial z_{j}$ is called of \emph{type $(1,0)$}, $Y(z) = \sum_{j=1}^{n}b_{j}(z)\partial/\partial\overline{z_{j}}$ is called of \emph{type $(0,1)$}. Note that 
$X\rho_{s}(z)=0,\; 1\leq s\leq k \Leftrightarrow \sum_{j=1}^{n}a_{j}(z)\partial\rho_{s}/\partial z_{j}(z) = 0\,, \,1\leq s\leq k \; \Leftrightarrow X(z) = (a_{1}, \hdots, a_{n}) \in T^{\mathbb{C}}_{z}M.$ Likewise, $Y\rho_s(z) = 0,\; 1\leq s\leq k\Leftrightarrow \overline{Y(z)} = (\overline{b_{1}}, \hdots, \overline{b_{n}}) \in T^{\mathbb{C}}_{z}M$.
In this situation, we say that $X \in T^{1,0}M$, and $Y\in T^{0,1}M$. We will sometimes abuse notation and use $T^{1,0}_{z}M$ (the fiber at $z$ of the bundle $T^{1,0}M$) in lieu of $T^{\mathbb{C}}_{z}M$ (the subspace of $T_{z}M$). They are naturally isomorphic. Care is needed, however; see items (ii) and (iii) in Lemma \ref{CR2} below.

\smallskip
We are now in a position to define \emph{CR-functions} and \emph{CR-mappings}.
Let $f$ be a function of class ${C}^1$ on a CR-submanifold $M$ of $\mathbb{C}^n$.
We say that $f$ is CR if and only if $\bar{L}f(z)=0$ for all  $z\in M$ and $L(z)\in T^{1,0}_zM$.
This set of equations is generally referred to as `the' (homogeneous) \emph{tangential Cauchy--Riemann equations}. CR-functions may be thought of as being holomorphic in complex tangential directions.
In particular, the simplest examples are restrictions of holomorphic functions; they satisfy the full Cauchy--Riemann equations, and so satisfy in particular the tangential ones. There are, however, CR-manifolds and CR-functions on them that at a particular point need not be extendible into any open neighborhood as an analytic function. The question whether or not CR-functions can be extended, perhaps to one side only, or only to a wedge, is highly nontrivial. We refer the reader to \cite{Boggess91, BER99, Tumanov96} and their references for more information on this interesting subject.

Let $M, N$ be CR-submanifolds of $\mathbb{C}^{n_{1}}$ and $\mathbb{C}^{n_{2}}$, respectively, and $f:M\to N$, $ z \mapsto (f_1(z), \dots f_{n_{2}}(z))$, be a ${C}^1$ map. $f$ is a \emph{CR-map} if $f_1,\dots, f_{n_{2}}$ are CR-functions on $M$.
The push forward $f_{*}$ (where $f$ is considered as a map between two real manifolds) maps $T_{z}M$ into $T_{f(z)}N$. If $f$ is a CR-map, the tangential Cauchy--Riemann equations for the functions $f_{j}$ show that $f_{*}$ (acting on $T_{z}M$) maps $T^{\mathbb{C}}_{z}M$ into $T^{\mathbb{C}}_{f(z)}N$, and this map is complex linear i.e $f_{*}$ commutes with $J$: $f_{*}|_{T_{z}^{\mathbb{C}}M}\circ J_{\mathbb{C}^{n_{1}}} = J_{\mathbb{C}^{n_{2}}}\circ f_{*}|_{T_{z}^{\mathbb{C}}M}$. Similarly, when $f_{*}$ acts on $T^{1,0}_{z}M\subset T_{z}M\otimes\mathbb{C}$,\footnote{See \cite{Boggess91}, section 7.1 for a discussion of the role of the complexified tangent bundle $TM\otimes\mathbb{C}$. In particular, note that $T_{z}^{1,0}M = T_{z}^{1,0}\mathbb{C}^{n}\cap T_{z}M\otimes\mathbb{C}$.} we have $f_{*}(T_{z}^{1,0}M)\subset T_{f(z)}^{1,0}N$. It is an elementary, but basic fact of CR-geometry that these properties characterize CR-maps:
\begin{lemma}\label{CR2}
The following are equivalent:

(i) $f$ is a CR-map.

(ii) $f_{*}(T_{z}^{\mathbb{C}}M) \subset T_{f(z)}^{\mathbb{C}}N$, and this map is $\mathbb{C}$-linear, for all $z\in M$.

(iii) $f_*({T_z^{1,0}M}) \subseteq T^{1,0}_{f(z)}N$, for all $z \in M$.
\end{lemma}
In (iii), $f_{*}$ is complex linear on ${T_z^{1,0}M}$ as well, but in contrast to (ii), we do not need to \emph{assume} that it is. To see that in (ii) it is necessary to make this assumption, consider a diffeomorphism $f$ from a complex submanifold to another. Then trivially $f_{*}(T_{z}^{\mathbb{C}}M)\subset T_{f(z)}^{\mathbb{C}}N$, since both equal the whole tangent space. Yet such an $f$ need not be holomorphic.
\begin{proof}[Proof of Lemma \ref{CR2}:]
The lemma is a combination of Definition 1, Lemma 1, and Theorem 1 of \cite{Boggess91}, section 9.2, where details are provided.
\end{proof}



Lemma \ref{CR2} immediately gives: 
\begin{lemma}\label{CR3}
Let  $M$ and $N$ as above such that the CR-dimensions of M and N are equal and let $f:M\to N $ be a CR-diffeomorphism. Then $f^{-1}$ is also CR. 
\end{lemma}
Lemma \ref{CR3} is false without the assumption that the CR-dimensions of $M$ and $N$ are equal. A simple example may be found in \cite{BER99}, page 50.

\smallskip

Finally, we define the class of CR-submanifolds that we are most interested in in this article. We say that \emph{$M$ is of hypersurface type} if the CR-codimension of $M$ is $1$. If $(m-1)$ is the CR-dimension of $M$, then the real dimension of $M$ is $(2m-1)$.

A CR-submanifold of hypersurface type $M$ can be represented in local coordinates as a graph over an actual hypersurface in $\mathbb{C}^{m}$, $m\leq n$, as follows. Fix $p \in M$. Choose coordinates so that $p$ corresponds to the origin, $(z_{1}, \hdots, z_{m-1})$ span $T^{\mathbb{C}}_{p}M$, the $x_{m}$-axis is tangent to $M$ at $p$ and (real) orthogonal to the  $(z_{1}, \hdots, z_{m-1})$ affine subspace, and the positive $y_{m}$-direction is given by $Je_{m}$. where $e_{m}$ is the unit vector in the positive $x_{m}$-direction. This is illustrated in the figure below, except that we have moved $p$ away from the origin in order to have a less crowded picture. 

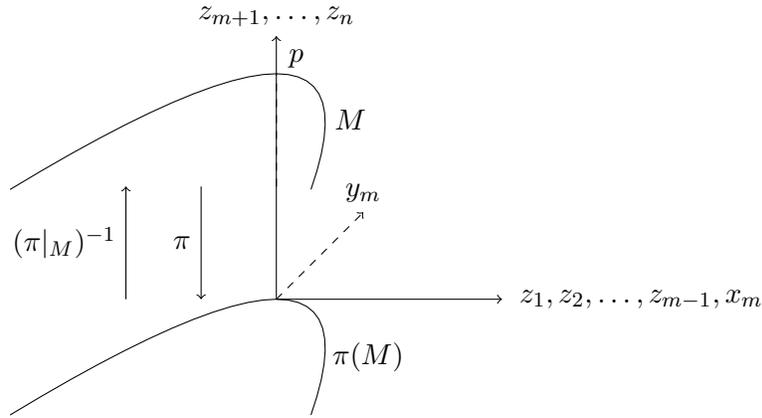
\vspace{0.15in}

\begin{center}
\begin{tikzpicture}
 \draw[->] (0,0,0) --  (0,3.5,0) node[above] {$z_{m+1},\dots, z_n$};
   \draw[domain= -2:2] plot (\x, 0, \x*\x);
   \node[left] at (2.6,0,2) {$\pi(M)$};
\node[left=0.2cm] at (0.7,3.2,0) {$p$};
\draw[domain=-2:2] plot (\x, 3,\x*\x);
  \draw[->] (0,0,0)  --  (3,0,0) node[right=0.1cm] {$z_1, z_2,\dots, z_{m-1}, x_m$};
 
   \draw[->,dashed] (0,0,0) --  (0,0,-3) node[above] {$y_m$}; 
   \draw[dashed] (0,1.5,0)--(0,3,0);
\node at (0.6, 2, -1) {$M$};
\draw[->] (-1, 1.5, 0) -- (-1,0,0) node[left, midway] {$\pi$};
\draw[->] (-2,0,0)--(-2, 1.5, 0) node[left, midway] {$(\pi|_{M})^{-1}$};
\end{tikzpicture}
\end{center}

\vspace{0.15in}

\noindent Now consider the projection $\pi : \mathbb{C}^n \longrightarrow \mathbb{C}^m$,
$(z_1,\dots, z_n) \longmapsto (z_1,\dots, z_m, 0,\dots, 0)$. $\pi(M)$ is a smooth hypersurface in $\mathbb{C}^{m}$, and $M$ is a graph over $\pi(M)$, given by 
\begin{equation}\label{map}
(\pi|_{M})^{-1}:\;(z_1,\dots, z_m, 0,\dots, 0)\mapsto (z_1,\dots, z_m, h_1(z_1,\dots,z_m),\dots, h_{n-m}(z_1,\dots,z_m))\;,
\end{equation}
where the functions $h_j$ are smooth on $\pi(M)$. However, more can be said: the functions $h_{j}$ are CR-functions on the hypersurface $\pi(M)$. Indeed, $\pi|_{M}$ is a CR-map (the projection is holomorphic) onto $\pi(M)$, and it is a diffeomorphism. By Lemma \ref{CR3}, so is its inverse. This means that the mapping functions are CR-functions. (Compare also \cite{Zampieri08}, Theorem 3.3.10.) This simple fact is essential in establishing the one-sided complexification in Theorem \ref{theo1} below.

\smallskip

From now on, we assume that $M$ is orientable. This means that there is a global section of unit length in $TM$ that is orthogonal to $T^{\mathbb{C}}_{z}M$ at each point $z$ of $M$. It is convenient to take this field, which we denote by $T$, to be purely imaginary (so that we are actually working in $TM\otimes\mathbb{C}$). The following Hermitian form on $T^{1,0}M$ plays an important role in CR-geometry. Let $p\in M$ and $X,Y\in T^{1,0}_pM$. Then
\begin{equation}\label{Levi}
 [X, \overline{Y}]_p=\lambda(X,\overline{Y})(p) T(p) \quad \text{mod}\; T^{1,0}_pM\oplus T^{0,1}_pM\;.
 \end{equation}
The Hermitian form $\lambda$ is \emph{the Levi form of $M$} (the reason for choosing $T$ purely imaginary is to make $\lambda$ Hermitian rather than skew Hermitian).

M is \emph{pseudoconvex} if $\lambda$ is positive semidefinite at all points of $M$, and
\emph{strictly pseudoconvex} if $\lambda$ is positive definite at all points of $M$. Should this happen with negative semidefinite and negative definite respectively, we can replace $T$ by $-T$ to revert to the former situation. These notions are of course generalizations of the corresponding notions for a hypersurface, where their significance is well known. It turns out that they are also appropriate for the problems discussed in this survey. Further elaborating on the meaning and significance of the Levi form and of pseudoconvexity would take us too far afield; again, \cite{Boggess91, BER99, Zampieri08} are excellent sources for further information. 

\smallskip

The main result of this section is the following theorem. The existence of the one-sided complexification is in \cite{Baracco1}, the second part comes from \cite{StraubeZeytuncu15}.
\begin{theorem}\label{theo1}
Let $M$ be an orientable, smooth, compact, pseudoconvex CR-submanifold of hypersurface type in $\mathbb{C}^n$. Then $M$ has a one-sided complexification to a complex submanifold $\widehat{M}$ of $\mathbb{C}^n$ (a `strip'), so that $M$ is the connected pseudoconvex component of the boundary of $\widehat{M}$. Moreover, if $U$ is any neighborhood of $M$, then near $M$, $\widehat{M}$ is contained in the hull of $M$ with respect to the functions that are plurisubharmonic in $U$.
\end{theorem}
\begin{proof}[Sketch of the proof of Theorem \ref{theo1}]
The situation is most elementary near a strictly pseudoconvex point, and we describe in detail how to construct the desired one-sided complexification locally. 

So let $p \in M$ be a strictly pseudoconvex point. Then $\pi(p)$ is a strictly pseudoconvex point of $\pi(M)$. Near $\pi(p)$, the pseudoconvex side of $\pi(M)$ can be filled by analytic discs with boundaries in $\pi(M)$. Moreover, any CR-function on $\pi(M)$ extends analytically, via extension along these discs, to this side. This is the Kneser--Lewy extension theorem, see for example \cite{Boggess91}, sections 14.1 and 15.2 (and \cite{Kneser36, Range02} for why this terminology is more appropriate than just `Lewy extension theorem').
In particular, the graphing (CR-) functions $h_{1}, \hdots, h_{(n-m)}$ extend holomorphically to the pseudoconvex side of $M$. These extensions, say $\widehat{h}_{1}, \hdots, \widehat{h}_{(n-m)
}$, then define an $m$-dimensional complex submanifold $\widehat{M}$ of $\mathbb{C}^{n}$ whose boundary near $p$ equals $M$:
\begin{equation}\label{extend}
\widehat{\pi^{-1}}: \;(z_{1}, \hdots, z_{m}) \longrightarrow \left(z_{1}, \hdots, z_{m}, \widehat{h}_{1}(z_{1}, \hdots, z_{m}), \hdots, \widehat{h}_{(n-m)}(z_{1}, \hdots, z_{m})\right) \;.
\end{equation}
$\widehat{\pi^{-1}}$ lifts the discs that fill out the pseudoconvex side of $\pi(M)$ near $\pi(p)$ to analytic discs with boundaries in $M$, and these analytic discs sweep $\widehat{M}$ near $p$. This geometric situation is illustrated in the figure below. 

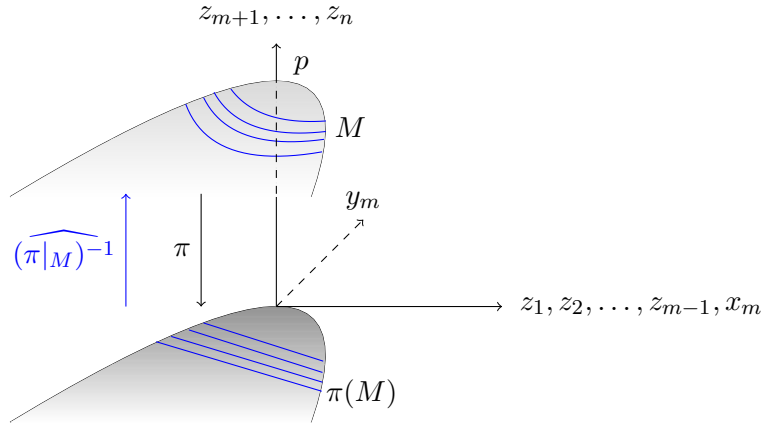
\begin{center}
\begin{tikzpicture}
  \draw[->] (0,0,0) --  (0,3.5,0) node[above=0.1cm] {$z_{m+1},\dots, z_n$};

 \draw[domain= -2:2] plot (\x, 0, \x*\x);
   \shade[top color=gray!85, bottom color=white, domain=-2:2] plot (\x, 0,\x*\x);
\draw[domain=-2:2] plot (\x, 3,\x*\x);
 \shade[top color=gray!30, bottom color=white, domain=-2:2] plot (\x, 3,\x*\x);
 \node[ right=0.1cm] at (0,3.2,0) {$p$};
  \draw[->] (0,0,0)  --  (3,0,0) node[right=0.1cm] {$z_1, z_2,\dots, z_{m-1}, x_m$};
   \draw[->, dashed] (0,0,0) --  (0,0,-3) node[above] {$y_m$}; 
   \draw[domain=-1.75:0.41, color=blue] plot (\x,2,-0.5-0.5*\x*\x);
   \draw[domain=-1.68:0.33, color=blue] plot (\x,2,-0.8-0.5*\x*\x);
   \draw[domain=-1.9:0.55, color=blue] plot (\x,2,-0.5*\x*\x);
    \draw[domain=-1.5:0.19, color=blue] plot (\x,2,-1.2-0.5*\x*\x);
   \node[left] at (2.9,0,3) {$\pi(M)$};
\draw[domain=-1.75:0.4, color=blue] (-1.75, -1, -0.5-0.5*1.75*1.75)--(0.4,-0.95, -0.5-0.5*0.4*0.4); 
\draw[domain=-1.75:0.4, color=blue] (-1.95, -1.1, -0.5-0.5*1.76*1.76)--(0.39,-1.1, -0.5-0.5*0.4*0.4); 
\draw[domain=-1.78:0.4, color=blue] (-2.2, -1.2, -0.5-0.5*1.78*1.78)--(0.39, -1.23, -0.5-0.5*0.4*0.4); 
\draw[domain=-1.78:0.4, color=blue] (-2.38, -1.25, -0.5-0.5*1.75*1.75)--(0.38, -1.35, -0.5-0.5*0.4*0.4); 
\draw[dashed] (0, 1.5,0)--(0,3,0);
\node at (0.6, 2, -1) {$M$};
\draw[->] (-1, 1.5, 0) -- (-1,0,0) node[left, midway] {$\pi$};
\draw[->, color=blue] (-2,0,0)--(-2, 1.5, 0) node[left, midway] {$\widehat{(\pi|_{M})^{-1}}$};
\end{tikzpicture}
\end{center}

\vspace{0.15in}
\noindent By the maximum principle for plurisubharmonic functions, $\widehat{M}$ is contained in the plurisubharmonic hull of $M$ with respect to the plurisubharmonic functions in $U$ (possibly after shrinking $\widehat{M}$).

Returning now to the general situation, one may first notice that one-sided extension of CR-functions of course holds under a much weaker assumption than strict pseudoconvexity. For $p\in M$, this property holds near $\pi(p)$ if (and only if) $\pi(M)$ contains no germ of a complex manifold of dimension $(m-1)$ through $\pi(p)$ (\cite{Trepreau06}). Equivalently, $M$ contains no germ of an $(m-1)$-dimensional complex manifold through $p$. Under this condition, one should be able to piece together such local extensions of $M$ to obtain a global extension. This strategy goes back to \cite{MunasingheStraube12} and works to construct $\widehat{M}$ (however, without the second statement in Theorem \ref{theo1}) in this case.
When $M$ does contain $(m-1)$-dimensional manifolds, and Tr\'{e}preau's theorem cannot be applied to obtain a local extension everywhere, this strategy runs into a nontrivial problem.
Baracco's crucial insight (\cite{Baracco1}) is that one can exploit the fact that extendibility of a CR-function propagates along analytic discs, in fact, along complex tangential curves (\cite{Baracco1, HangesTreves83, Tumanov95, Tumanov96}), to overcome this problem. Making this idea precise takes some work. In particular, care has to be taken to use extension results based on analytic disc methods in order to obtain the second statement in Theorem \ref{theo1} (via the maximum principle).

A sketch is as follows. First, note that $M$ consists of a single CR-orbit; in other words, any two points on $M$ can be connected by a piecewise smooth curve whose tangent at each point $z$ belongs to $T^{\mathbb{C}}_{z}M$ (\cite{BerhanuMendoza97}, Corollary 4.4, \cite {Baracco1}, Proposition 2.1). Next, pick and fix a strictly pseudoconvex point $z_{0}\in M$. We can always find such point by enclosing $M$ in a large sphere and shrinking the radius until the sphere touches $M$. Then, near $z_{0}$, we have a local extension of $M$ as above. Now take an arbitrary point $z_{1}$ on $M$, and connect $z_{0}$ to $z_{1}$ by a 
piecewise smooth CR-curve, say $\gamma(t)$, $0\leq t \leq 1$. 
Using extendibility results from \cite{BER99}, chapter 8 and from \cite{Tumanov95}, which are all based on analytic disc techniques, one can now show that this local extension property of $M$, including the fact that near $M$, the extension is contained in the (local) plurisubharmonic hull of $M$,  propagates along $\gamma$ from $z_{0}$ to $z_{1}$. There are numerous technicalities to take care of. For these, we refer the reader to \cite{Baracco1, StraubeZeytuncu15}.
\end{proof}
Once we have the `strip' manifold $\widehat{M}$, we can say more. Namely, as long as we are willing to `go outside $\mathbb{C}^{n}$', $M$ actually bounds a complex manifold in the $C^{\infty}$ sense \cite{Baracco2, Baracco3}).

\begin{corollary}\label{cor1}
Assume $M$ is as in Theorem \ref{theo1}. Then $M$ is the boundary, in the $C^{\infty}$ sense, of a complex manifold (which need not be a submanifold of $\mathbb{C}^{n}$).
 \end{corollary}
\begin{proof}[Proof (from \cite{Baracco2, Baracco3}).]
Consider the one-sided complexification (the `strip') $\widehat{M}$ of $M$ provided by Theorem \ref{theo1}.
We would like to continue $\widehat{M}$ to a complex variety. This situation is classical and belongs to the circle of ideas studied by the German school of several complex variables in the fifties and sixties, see  
\cite{Rothstein59}, \cite{RothsteinSperling66}. It turns out that this extension can indeed be achieved, albeit not necessarily within $\mathbb{C}^{n}$. That is, there is a complex variety $\mathcal{M}$ which extends the strip $\widehat{M}$ (\cite{Rothstein59}, Satz 2 in section 4 and its proof). A sketch of a proof can also be found in \cite{Baracco3}, section 3. Because $\mathcal{M}$ extends $\widehat{M}$, its singularities stay away from its boundary $M$ and thus form a compact analytic set. Therefore, they are isolated, and there are only finitely many. By blowing 
them up, we get the desired complex manifold.
\end{proof}
The variety in the above proof is not to be confused with the Harvey--Lawson variety from \cite{HarveyLawson75}, which is \emph{immersed} in $\mathbb{C}^{n}$ such that $M$ is its boundary in the sense of currents. Because it is only immersed, its singularities need not stay away from its boundary $M$. For further clarification of these matters, see \cite{LukYau98, Fefferman00, HarveyLawson00, Baracco3}.

\section{The $L^2$-theory for the $\bar\partial_M$-complex}\label{L-2}

Let $M$ be a CR-submanifold in $\mathbb{C}^n$ of hypersurface type that is orientable, smooth, and closed. In this section, we briefly recall the $\overline{\partial}_{M}$-complex in the context of $L^{2}$ spaces of forms, and we explain the main fact, namely that $\overline{\partial}_{M}$ has closed range in $L^{2}$. We then recall the basic facts about the complex Green operator that follow. Further information and background for this section may be found in \cite{Boggess91, BER99, Zampieri08, Shaw85, ChenShaw01, Nicoara06}.

\smallskip

We denote by $\Lambda^{0,q}T^*(\mathbb{C}^n)$ the bundle of smooth $(0,q)$-forms on $\mathbb{C}^n$ and by $\Lambda^{0,q}T^*(\mathbb{C}^n)|_{M}$ its coefficient wise restriction to $M$. This restriction is to be distinguished from the restriction as forms. 

Locally, $M$ is defined by $k$ defining functions $\rho_{1}, \hdots, \rho_{k}$, where $k=2n-2m+1$ is the (real) codimension of $M$ in $\mathbb{C}^{n}$, and the gradients $\nabla\rho_{1}(z), \hdots, \nabla\rho_{k}(z)$ are linearly independent over $\mathbb{R}$ at points $z\in M$. Denote by $I^{0,q}$ the ideal generated by $\lbrace{\overline{\partial}\rho_1,\dots,\overline{\partial}\rho_k}\rbrace$.
Then $\Lambda^{0,q}T^*M$ defined by $\Lambda^{0,q}T^*M(z):=\left((I^{0,q})|_{M}(z)\right)^{\perp}$ (inside $\Lambda^{0,q}T^*(\mathbb{C}^n)|_{M}(z)$) is \emph{the bundle of $(0,q)$-forms on $M$}. The orthogonal complement is taken with respect to the standard pointwise inner product for forms that declares the $d\overline{z_{J}}$, $|J|=q$, to form an orthonormal basis. Because $M$ is CR,
the dimension of the (complex) span of $\{\nabla\rho_{1}(z), \hdots, \nabla\rho_{k}(z)\}$ is independent of $z$, so that we have indeed defined a vector bundle locally. However, $\left((I^{0,q})_{\mid_M}(z)\right)^{\perp}$ is independent of the choice of defining functions, so that $\Lambda^{0,q}T^*M$ is well defined globally. 

Let $t_M: \Lambda^{0,q}T^*(\mathbb{C}^n)|_{M}\rightarrow \Lambda^{0,q}T^*M$ be the orthogonal projection. For a form $f$, we will often use the notation $f_{t_{M}}$ for $t_{M}f$. $f_{t_{M}}$ (or $t_{M}f$) is called the \emph{tangential part} of $f$. 

We are now ready to define the \emph{tangential Cauchy--Riemann operator} and the associated \emph{$\overline{\partial}_{M}$-complex} $: \Lambda^{0,q}T^*M\rightarrow \Lambda^{0,q+1}T^*M$.
Let $f\in\Lambda^{0,q}T^*M$ and $\widetilde{f}$ be an extension of $f$ into a neighborhood of $M$ (by extending each coefficient function). Then we set
\begin{equation}\label{dbarM}
\bar\partial_Mf=(\bar\partial \widetilde{f})_{t_M}\;.
\end{equation}
$\bar\partial_Mf$ is well defined; one checks that if $\widetilde{f_{1}}$ and $\widetilde{f_{2}}$ are two extensions of $f$, then 
$\left(\bar\partial(\widetilde{f_1}-\widetilde{f_2})\right)_{t_M}$ vanishes.  $\overline{\partial}_{M}$ inherits from $\overline{\partial}$ the property that it forms a complex, that is, 
$\bar\partial_M\circ\bar\partial_M=0$. Similarly, it inherits the 'usual' rules of computing with differential forms.

\smallskip

It is useful to have an expression for $\overline{\partial}_{M}$ in a local coordinate frame.
Let $p\in M$ and $L_1,\dots, L_{m-1}$ be a local orthonormal basis for $T^{1,0}M$ in a neighborhood of $p$ and $\omega_1,\dots, \omega_{m-1}$ its dual frame. In this basis, a $(0,q)$-form $f$ on $M$ can be written as
$\sideset{}{'}\sum_{\vert J\vert=q}f_J\overline{\omega_J},$
where $\overline{\omega_J}=\overline{\omega_{j_1}}\wedge\dots\wedge\overline{\omega_{j_q}}$ and where the notation $\sideset{}{'}\sum$ means the summation over strictly increasing multi-indices. Note that in such a frame, the pointwise inner product between forms becomes
\begin{equation}\label{inner}
(f,g)(z) = \sideset{}{'}\sum_{|J|=q}f_{J}(z)\overline{g_{J}(z)}\,.
\end{equation}
For $\overline{\partial}_{M}$, we have
\begin{equation}\label{local}
\bar\partial_Mf=\sum_{j=1}^{m-1}\sideset{}{'}\sum_{\vert J\vert=q}(\overline{L_j}f_J)\overline{\omega_j}\wedge \overline{\omega_J}+\sideset{}{'}\sum_{\vert J\vert=q}f_J\bar\partial_M\overline{\omega_J}.
\end{equation}
It suffices to check this formula for $(0,0)$-forms (functions); it then follows by the product rule for $(0,q)$-forms. So take a function $g$ on $M$ and extend it to $\widetilde{g}$; then $\overline{\partial}_{M}g = (\overline{\partial}\widetilde{g})_{t_{M}}$. Locally, we can extend $L_{1}, \hdots, L_{m-1}$ to an orthonormal basis of $T\mathbb{C}^{n}$, with dual basis $\omega_{1}, \hdots, \omega_{n}$. In this basis, $\overline{\partial}\widetilde{g} = \sum_{j=1}^{n}\overline{L_{j}}(\widetilde{g})\overline{\omega_{j}}$. The tangential part of this form, on $M$, is $\sum_{j=1}^{m-1}\overline{L_{j}}g$, as desired.

\smallskip

The pointwise inner product on forms yields an $L^{2}$ inner product on $\Lambda^{0,q}T^{*}M$:
\begin{equation}\label{inner1}
(f,g)_{L^{2}_{(0,q)}(M)}:=\int_{M}(f,g)(z)d\sigma(z) \;,
\end{equation}
where $\sigma$ is the induced Lebesgue measure on $M$.
We denote by $L^2_{(0,q)}(M)$ the Hilbert space obtained by completing $\Lambda^{0,q}T^*M$ under this inner 
product.

\smallskip

We extend the tangential Cauchy--Riemann operator to an unbounded operator $\overline{\partial}_M: L^2_{(0,q)}(M)\rightarrow L^2_{(0,q+1)}(M)$ with the maximal domain of definition. That is, we set
$dom(\bar\partial_M)=\lbrace{f\in L^2_{(0,q)}(M) \mid\; \bar\partial_Mf\in L^2_{(0,q+1)}(M)}\rbrace$ where $\bar\partial_M$ is computed in the sense of distributions, say in local frames as in \eqref{local} (whether or not the result is in $L^{2}$ does not depend on the local frame chosen). 
$\left(L^2_{(0,q)}(M), \bar\partial_M\right)$ gives the $L^{2}$-$\overline{\partial}_{M}$-complex:
\begin{equation}\label{dbarMcomplex}
L^2(M)\overset{\bar\partial_M}{\rightarrow}L^2_{(0,1)}(M)\overset{\bar\partial_M}{\rightarrow}L^2_{(0,2)}(M)\overset{\bar\partial_M}{\rightarrow}\dots\overset{\bar\partial_M}{\rightarrow}L^2_{(0,m-1)}(M)\overset{\bar\partial_M}{\rightarrow} 0\;.
\end{equation}

\smallskip

At each form level, $\bar\partial_M$ is a closed, densely defined operator. As such, it has a Hilbert space adjoint, denoted by $\bar\partial_M^*$. In local coordinates, \eqref{local},\eqref{inner}, and integration by parts give
\begin{equation}\label{adjoint}
\bar\partial_M^*f=-\sum_{j=1}^{m-1}\sideset{}{'}\sum_{\vert K\vert=q-1}L_jf_{jK}\overline{\omega_K}+ \text{terms of order zero}\;,
\end{equation}
with $jK=(j, k_1,\dots,k_{q-1})$. The coefficients $f_{J}$ are defined for \emph{all} $q$-tuples $(j_{1}, \hdots, j_{q})$ in the usual alternating manner via $f_{J}\,\overline{\omega_{J}} = f_{\widetilde{J}}\;\overline{\omega_{\widetilde{J}}}$, where $\widetilde{J}$ is the increasing rearrangement of $J = (j_{1}, \hdots, j_{q})$. `Terms of order zero' refers to terms where the coefficients of $f$ are not differentiated. It is worthwhile to note that in contrast to the situation for $\overline{\partial}^{*}$ on a domain, there does not enter a boundary condition ($M$ has no boundary). On the other hand, the complex $\overline{\partial}\oplus\overline{\partial}^{*}$ is elliptic, while $\overline{\partial}_{M}\oplus\overline{\partial}_{M}^{*}$ is not (the direction complementary to the complex tangent space is not under control). 

\smallskip

In order to have a well behaved $L^{2}$-theory and an associated Hodge decomposition, one needs the range of $\overline{\partial}_{M}$ to be closed in $L^{2}$. This is indeed the case (\cite{Nicoara06, Baracco1, Baracco2, HarringtonRaich10}).
\begin{theorem}\label{closedrange}
Let $M$ be a compact smooth orientable pseudoconvex CR-submanifold of $\mathbb{C}^n$ of hypersurface type.
The tangential Cauchy--Riemann operator $\bar\partial_M:L^2_{(0,q)}(M)\rightarrow L^2_{(0,q+1)}(M)$ has closed range for $0\leq q\leq m-1$.
\end{theorem}
Shaw (for $0\leq q\leq m-3$, \cite{Shaw85}) and
Boas and Shaw (for $q=m-2$, \cite{BoasShaw86}) proved the closed range property of $\bar\partial_M$ in $L^2$ when $M=\partial\Omega$ is the smooth boundary of a pseudoconvex domain $\Omega$ relatively compact in $\mathbb{C}^n$. Kohn (\cite{Kohn86}), using different (namely microlocal) methods, generalized their result for the boundary $M$ of a pseudoconvex manifold. In 2006, Nicoara (\cite{Nicoara06}) proved Theorem \ref{closedrange} for $m\geq 2$, also using microlocal methods. Baracco's proof (\cite{Baracco1, Baracco2}) is different; it is based on his result given in Corollary \ref{cor1}; this result reduces the general case to that considered in \cite{Kohn86}. When the CR-dimension of $M$ is at least two, Baracco also gives an alternative proof based on Theorem \ref{theo1} that does not rely on Corollary \ref{cor1} (\cite{Baracco1}). 
Harrington and Raich (\cite{HarringtonRaich10}) prove these closed range properties and related results under a condition they call weak $Y(q)$; this is a condition weaker than pseudoconvexity, adapted to the form level $q$ at which estimates are considered (their work also requires $m\geq 2$).

\smallskip

With Theorem \ref{closedrange} in hand, elementary Hilbert space theory gives the following Hodge decomposition:
\begin{equation}\label{Hodge}
L^{2}_{(0,q)}(M) = \ker (\overline{\partial}_{M}) \oplus \ker(\overline{\partial}_{M}^{*}) \oplus \mathcal{H}_{q}(M)\;,
\end{equation}
where $\mathcal{H}_{q}(M)$ denotes the space of so called harmonic forms, $\mathcal{H}_{q}(M) = \ker(\overline{\partial}_{M})\cap \ker(\overline{\partial}_{M}^{*})$. To this orthogonal decomposition, there corresponds the estimate
\begin{equation}\label{Hodge2}
\Vert u\Vert^2_{L^2_{(0,q)}(M)}\lesssim \Vert \bar\partial_Mu\Vert^2_{L^2_{(0,q+1)}(M)}+ \Vert \bar\partial^*_Mu\Vert^2_{L^2_{(0,q-1)}(M)}+\Vert H_qu\Vert^2_{L^2_{(0,q)}(M)}\;,
\end{equation}
where $H_q:L^2_{(0,q)}(M)\to \mathcal{H}_q(M)$ is the orthogonal projection. \eqref{Hodge2} has to be interpreted in the obvious way when $q=0$ or $q=(m-1)$.
We summarize the relevant Hilbert space geometry with the following figure:

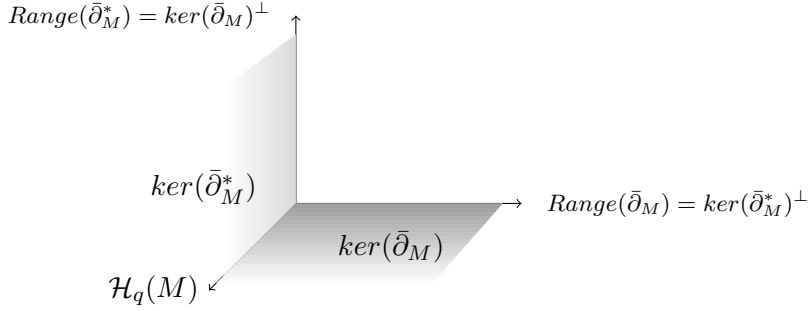
\begin{center}
\hspace{-0.3in}
\begin{tikzpicture}

 \draw[->] (0,0,0) --  (0,0,3) node[left] {$\mathcal{H}_q(M)$};
 \draw[->] (0,0,0) -- (3,0,0);
 \draw[->] (0,0,0) -- (0,2.5,0); 


 \shade[right color=gray!30, left color=white] (0,0,0)--(0,2.25,0)--(0,2.5,2.5)--(0,0,2.5) ; 
  \node[left] at (0,0.6,1) {$ker(\bar\partial^*_M)$};
 \shade[top color=gray!80, bottom color= white] (0,0, 0)--(2.75,0,0)--(2.875,0,2.75)--(0,0,2.75);
  \node[right] at (1,0,1.5) {$ker(\bar\partial_M)$};
\node[left=0.2cm] at (0,2.5) {\footnotesize{$Range(\bar\partial^*_M)=ker(\bar\partial_M)^{\perp}$}};
\node[right=0.2cm] at (3,0) {\footnotesize{$Range(\bar\partial_M)=ker(\bar\partial^*_M)^{\perp}$}};

 \end{tikzpicture}
\end{center}


The complex Laplacian, or Kohn Laplacian, $\Box_{M,q}$ is defined by
\begin{equation}
\Box_{M,q}: L^2_{(0,q)}(M)\longrightarrow L^2_{(0,q)}(M), \;\Box_{_M,q} :=\bar\partial_M\bar\partial^*_M+\bar\partial^*_M\bar\partial_M.
\end{equation}
The domain of $\Box_{M,q}$ is understood to be the set of forms where this expression makes sense. Equivalently, $\Box_{M,q}$ is the unique self-adjoint operator associated to the Hermitian form $Q_{q}(f,f) = (\overline{\partial}_{M}f, \overline{\partial}_{M}f) + (\overline{\partial}_{M}^{*}f, \overline{\partial}_{M}^{*}f)$
 via
 \begin{equation}\label{QM1}
 Q_{q}(f,f) = (\Box_{M,q}f,f)_{L^{2}_{(0,q)}(M)} \;, \; f\in dom(\Box_{M,q})\;.
 \end{equation}
The discussion here is analogous to the one for $\Box_{q}$ associated to $\overline{\partial}$ on a domain in $\mathbb{C}^{n}$; details for that case may be found in \cite{Straube10a}, section 2.8.

It is immediate from \eqref{QM1} that $\ker(\Box_{M,q}) = \mathcal{H}_{q}$. Moreover, because the range of $\overline{\partial}_{M}$ is closed (and thus so is the range of $\overline{\partial}_{M}^{*}$), the range of $\Box_{M,q}$ is also closed. Consequently, $\Box_{M,q}$ maps $\mathcal{H}_{q}^{\perp}$ injectively onto itself. The (bounded) inverse is called \emph{the complex Green operator $G_{q}$}. We can easily make the boundedness of $G_{q}$ explicit; \eqref{Hodge2} and \eqref{QM1} give
\begin{multline}
 \;\;\;\;\; \|f\|_{L^{2}_{(0,q)}(M)}^{2} \lesssim \|\overline{\partial}_{M}f\|_{L^{2}_{(0,q+1)}(M)}^{2} + \|\overline{\partial}_{M}^{*}f\|_{L^{2}_{(0,q-1)}(M)}^{2} \\= \left(\Box_{M,q}f,f\right)_{L^2_{(0,q)}(M)}
 \leq \|\Box_{M,q}f\|_{L^2_{(0,q)}(M)}\|f\|_{L^2_{(0,q)}(M)}\;,\;\;\;\;\;\end{multline}
for $f\in dom(\Box_{M,q})\cap \mathcal{H}_{q}^{\perp}$. Consequently, $\|f\|_{L^2_{(0,q)}(M)}\lesssim \|\Box_{M,q}f\|_{L^2_{(0,q)}(M)}$, for $f\in dom(\Box_{M,q})\cap\mathcal{H}_{q}^{\perp}$.
It is convenient and customary to set $G_{q}$ equal to zero on $\mathcal{H}_{q}$, so that it becomes defined on all of $L^{2}_{0,q}(M)$.

The study of the Green operator is of interest for a number of reasons. We only mention one example; if one `knows' $G_{q}$, one can solve the inhomogeneous tangential Cauchy--Riemann equation $\bar\partial_M u=f$. Indeed, if $f$ satisfies the two obvious necessary conditions for solvability, that is $f\in\mathcal{H}_q^{\perp}$ and $f$ is $\overline{\partial}_{M}$-closed, then 
\begin{equation}\label{invert}
f=\bar\partial_M\bar\partial^*_MG_q f+\bar\partial^*_M\bar\partial_M G_q f\;.
\end{equation}
Since $\bar\partial_M f=0$,  $\bar\partial^*_M\bar\partial_M G_q f$ must also be $\bar\partial_M$-closed (by \eqref{invert}). But 
$\bar\partial^*_M\bar\partial_M G_q f$ is also orthogonal to the kernel of $\overline{\partial}_{M}$. Therefore, it must vanish,
and \eqref{invert} implies $f=\bar\partial_M(\bar\partial^*_MG_q f)$. The solution $u = \bar\partial^*_MG_q f\in ker(\bar\partial_M)^{\perp}$ is called the \emph{minimal solution} (it minimizes the $L^{2}$-norm) or the \emph{Kohn solution}.

\smallskip

Let $j_q: \mathcal{H}_q(M)^{\perp}\cap dom(\bar\partial_M)\cap dom(\bar\partial_M^*)\hookrightarrow \mathcal{H}_q(M)^{\perp}$ be the imbedding, and set 
\begin{equation}\label{graph}
\Vert u\Vert_{graph}^2:= Q_{q}(u,u) = \Vert \bar\partial_M u\Vert^2_{L^2_{(0,q+1)}(M)}+\Vert \bar\partial^*_M u\Vert^2_{L^2_{(0,q-1)}(M)}\;.
\end{equation}

 With this norm, $\mathcal{H}_q(M)^{\perp}\cap dom(\bar\partial_M)\cap dom(\bar\partial_M^*)$ becomes a Hilbert space (in view of \eqref{Hodge2}). \eqref{Hodge2} then also says that $ j_q$ is continuous, hence so is its adjoint $j_{q}^{*}$. The following simple expression for $G_{q}$ is useful when studying compactness of $G_{q}$. It is the analogue of the expression given for the $\overline{\partial}$-Neumann operator in part (1) of Theorem 2.9 in \cite{Straube10a}.
\begin{lemma}\label{Green}
\begin{equation}
(G_q)|_{\mathcal{H}_{q}^{\perp}} =j_q\circ (j_q)^*\;.\;\;\;\;\;\;\;\;\;\;\;\;\;\;\;\;\;\;\;\;\;\;\;\;\;\;\;\;\;\;\;\;\;\;\;\;\;\;\;\;\;
\end{equation}
\end{lemma}
\begin{proof}
Let $u\in \mathcal{H}_{q}(M)^{\perp}$ and $v\in dom(\bar\partial_M)\cap dom(\bar\partial_M^*)\cap \mathcal{H}_q(M)^{\perp}$. Then
\begin{equation}
(u,v)_{L^2_{(0,q)}(M)} = (u, j_q v)_{L^2_{(0,q)}(M)} = ((j_q)^*u,v)_{graph}\; ,
\end{equation}
where the inner product on the right-hand side is the inner product on $\mathcal{H}_q(M)^{\perp}\cap dom(\bar\partial_M)\cap dom(\bar\partial_M^*)$ corresponding to the graph norm indicated above (and so equals $Q_{q}(j_{q}^{*}u,v)$). On the other hand, by virtue of \eqref{QM1},
\begin{equation}
(u,v)_{L^2_{(0,q)}(M)} = (\square_qG_qu,v)_{L^2_{(0,q)}(M)} = (G_q u,v)_{graph}\;.
\end{equation}
Equating the right-hand sides of the previous two equations shows that as a map into  $\mathcal{H}_q(M)^{\perp}\cap dom(\bar\partial_M)\cap dom(\bar\partial_M^*)$, $G_{q}$ equals $(j_{q})^{*}$. Note that $G_{q}$ maps into $\mathcal{H}_{q}(M)^{\perp}\cap dom(\Box_{M,q}) \subseteq \mathcal{H}_q(M)^{\perp}\cap dom(\bar\partial_M)\cap dom(\bar\partial_M^*)$. Therefore, as a map into $\mathcal{H}_{q}^{\perp}$, $G_{q}$ equals $j_{q}\circ (j_{q})^{*}$.
\end{proof}

It is natural to ask whether the dimension of $\mathcal{H}_{q}(M)$ is finite. This is indeed the case and was shown in \cite{Nicoara06} and in \cite{HarringtonRaich10} (in a more general context). That is $dim (\mathcal{H}_{q}(M)) < \infty$ when $1\leq q\leq (m-2)$. In general, this dimension will be non-zero. Examples can be obtained for instance based on an observation of Brinkschulte in \cite{Brink02}. Namely, there are even strictly pseudoconvex submanifolds of $\mathbb{C}^{n}$ of hypersurface type whose smooth $\overline{\partial}_{M}$-cohomology at the level of $(0,1)$-forms is nontrivial. For these examples, one can show that the $L^{2}$- $\overline{\partial}_{M}$-cohomology is also nontrivial; see the discussion in \cite{StraubeZeytuncu15} following Corollary 1.
In other words, in these examples, $dim(\mathcal{H}_{1}(M)) \neq 0$.

The finite dimensionality of $\mathcal{H}_{q}(M)$ is reflected in a version of the basic $L^{2}$ estimate \eqref{Hodge2} that no longer contains the harmonic component (\cite{StraubeZeytuncu15}, estimate (7)).

\begin{lemma}\label{basiclemma}
Let $M$ be as above. Then
\begin{multline}\label{basic}
\;\;\;\;\Vert u\Vert^2_{L^2_{(0,q)}(M)}\lesssim \Vert \bar\partial_Mu\Vert^2_{L^2_{(0,q+1)}(M)}+ \Vert \bar\partial^*_Mu\Vert^2_{L^2_{(0,q-1)}(M)}+\Vert u\Vert^2_{W^{-1}_{(0,q)}(M)}\;, \\
u \in dom(\overline{\partial}_{M})\cap dom(\overline{\partial}_{M}^{*})\;, 1\leq q\leq (m-2)\;.\;\;\; \; 
\end{multline}
\end{lemma}
\begin{proof}
We reproduce the short argument from \cite{StraubeZeytuncu15}.
Because $\mathcal{H}_{q}(M)$ is finite dimensional for $1\leq q\leq (m-2)$, any two norms on it are equivalent. In particular
\begin{equation}\label{equivalent}
 \|u\|_{L^{2}_{(0,q)}(M)} \lesssim \|u\|_{W^{-1}_{(0,q)}(M)}\;,\;u \in \mathcal{H}_{q}(M)\;.
\end{equation}
For $u \in dom(\overline{\partial}_{M})\cap dom(\overline{\partial}_{M}^{*})$ we therefore have
\begin{multline}\label{mult1}
 \|H_{q}u\|_{L^{2}_{(0,q)}(M)}^2 \lesssim \|H_{q}u\|_{W^{-1}_{(0,q)}(M)}^2 \leq \|u\|_{W^{-1}_{(0,q)}(M)}^2 + \|H_{q}u - u\|_{W^{-1}_{(0,q)}(M)}^2\\ 
 \lesssim \|u\|_{W^{-1}_{(0,q)}(M)}^2 + \|H_{q}u - u\|_{L^{2}_{(0,q)}(M)}^2 \lesssim \|u\|_{W^{-1}_{(0,q)}(M)}^2 + \|\overline{\partial}_{M}u\|_{L^{2}_{(0,q+1)}(M)}^2 + \|\overline{\partial}_{M}^{*}u\|_{L^{2}_{(0,q-1)}(M)}^2\;.
 \end{multline}
In the last inequality, we have used \eqref{Hodge2} for $H_{q}u - u$ together with $H_{q}u \in ker(\overline{\partial}_{M}) \cap ker(\overline{\partial}_{M}^{*})$ and $H_{q}(H_{q}u - u) = 0$. Substituting \eqref{mult1} into \eqref{Hodge2} gives \eqref{basic}.
\end{proof}
Estimate \eqref{basic} is the form of the basic $L^{2}$ estimate that we will use in sections \ref{Compact} and \ref{Sobolev}. Note that in deriving \eqref{equivalent} from finite dimensionality of $\mathcal{H}_{q}(M)$, we do not have control of the constant. But of course, this finite dimensionality is established via estimates, so that one can establish control of the constant from there. In fact, \eqref{equivalent} is derived explicitly in \cite{HarringtonRaich10} in the proof of Lemma 5.1, which then, among other things, implies finite dimensionality of $\mathcal{H}_{q}(M)$ (because $L^{2}_{(0,q)}(M)$ imbeds compactly into $W^{-1}_{(0,q)}(M)$, the estimate implies that the identity on $\mathcal{H}_{q}(M)$, with the $L^{2}$-norm, is compact) . 

\section{Compactness estimates for the complex Green operator}\label{Compact}

In this section, we keep the previous notations, but we drop the requirement that $M$ is orientable: compactness estimates are a local property, and so one would not expect orientability to play a role. Pseudoconvexity in this case is defined locally. The vector field $T$ used in section \ref{CRsubmanifolds} to define the Levi form still exists locally, providing a definition of the Levi form locally. $M$ is \emph{pseudoconvex} if every point has a neighborhood where the Levi form does not change sign. Throughout the section, $M$ is a smooth compact pseudoconvex CR-submanifold in $\mathbb{C}^n$ of hypersurface type (without boundary), 

\smallskip

The reasons for studying compactness of the complex Green operator are somewhat analogous to those for studying compactness of the $\overline{\partial}$-Neumann operator; see \cite{Straube10a}, introduction to chapter 4 for a discussion. In particular, compactness estimates for $G_{q}$ (or $\overline{\partial}_{M}$) imply regularity in Sobolev spaces.

We first reformulate the property that $G_{q}$ is compact in terms of a family of estimates. These estimates are usually referred to as \emph{compactness estimates} or sometimes just as \emph{a} compactness estimate. They take the following form: for all $\varepsilon >0$, there exists a positive constant $C_{\varepsilon}$ such that
\begin{equation}\label{compest2} 
\Vert u\Vert_{L^2_{(0,q)}(M)}^2\leq \varepsilon\left( \Vert \bar\partial_Mu\Vert^2_{L^2_{(0,q+1)}(M)}+ \Vert \bar\partial^*_Mu\Vert^2_{L^2_{(0,q-1)}(M)}\right)+C_\varepsilon\Vert u\Vert^2_{W^{-1}_{(0,q)}(M)}\;.
\end{equation} 
Before proceeding, we note that the family of compactness estimates \eqref{compest2} implies the $L^{2}$-theory as discussed above; in particular, $\overline{\partial}_{M}$ and $\overline{\partial}_{M}^{*}$ have closed range (so that \eqref{Hodge2} holds), and $\mathcal{H}_{q}(M)$ is finite dimensional for $1\leq q\leq (m-2)$ (see for example \cite{Hormander65}, Theorems 1.1.3 and 1.1.2). Of course, we trivially have \eqref{basic}  by fixing some value of $\varepsilon$ in \eqref{compest2}.
\begin{lemma}\label{compest}
Let $1\leq q\leq (m-2)$. The following properties are equivalent:
\begin{enumerate}
\item $G_q$ is compact.
\item $j_{q}$ is compact.
\item $\overline{\partial}_{M}$ satisfies `the' compactness estimate \eqref{compest2}.
\end{enumerate}
\end{lemma}
\begin{proof} $G_{q}$ is compact if and only if $G_{q}|_{\mathcal{H}_{q}(M)^{\perp}}$ is compact. The equivalence of (1) and (2) is thus immediate from Lemma \ref{Green}. The equivalence of (2) and (3) is essentially a consequence of Lemma \ref{fual} below.

Assume that (3) holds, let $u \in dom(\overline {\partial}_{M})\cap dom(\overline{\partial}_{M}^{*})\cap\mathcal{H}_{q}(M)^{\perp}$, and fix $\varepsilon >0$. Then (3) says
\begin{equation}
 \|j_{q}u\|^{2}_{L^{2}_{(0,q)}(M)} = \|u\|_{L^{2}_{(0,q)}(M)}^{2} \leq \varepsilon\|u\|_{graph}^{2} + C_{\varepsilon}\|u\|_{W^{-1}_{(0,q)}(M)}^{2}\;.            
\end{equation}
Because  \eqref{Hodge2} holds when (3) holds, $dom(\overline{\partial}_{M})\cap dom(\overline{\partial}_{M}^{*})\cap\mathcal{H}_{q}(M)^{\perp}$ imbeds continuously into $L^{2}_{(0,q)}(M)$, hence compactly into $W^{-1}_{0,q}(M)$. Applying Lemma \ref{fual}, part (2), with $X = dom(\overline{\partial}_{M})\cap dom(\overline{\partial}_{M}^{*})\cap\mathcal{H}_{q}(M)^{\perp}$ (endowed with the graph norm), $Y = L^{2}_{(0,q)}(M)$, $Z_{\varepsilon} = W^{-1}_{(0,q)}(M)$ for all $\varepsilon$, and $T$ and $S_{\varepsilon}=S$ the obvious embeddings, gives (2). On the other hand, when (2) holds, part (1) of Lemma \ref{fual} with $X$, $Z$, $T$, and $S$ as above, gives (3).
\end{proof}
The following lemma from functional analysis appears in various versions in the literature. The complete version given here appears in \cite{Straube10a}, Lemma 4.3, where one also finds a proof. (Note that in the proof of Lemma \ref{compest} we did not use the full strength of (2); the space $Z_{\varepsilon}$ was the same for all $\varepsilon$). 
\begin{lemma}\label{fual}
Let $X,Y$ be Hilbert spaces and $T: X \to Y$ a linear operator.
\begin{enumerate}
\item Let $Z$ be another Hilbert space and let $S: X \to Z$ be a continuous linear operator that is injective.
If $T$ is compact, then for all $\varepsilon>0$, there exists $C_\varepsilon>0$ such that 
$$\Vert Tu\Vert^2_Y\leq \varepsilon \Vert u\Vert^2_X+C_\varepsilon\Vert Su\Vert^2_Z.$$
\item Assume that for all $\varepsilon>0$, there exist a Hilbert space $Z_\varepsilon$, a linear compact operator $S_\varepsilon: X \to Z_\varepsilon$, and a constant $C_\varepsilon>0$ such that 
$$\Vert Tu\Vert^2_Y\leq \varepsilon \Vert u\Vert^2_X+C_\varepsilon\Vert S_\varepsilon u\Vert^2_{Z_\varepsilon}.$$
Then $T$ is compact.
\end{enumerate}
\end{lemma}

\smallskip

Next, we make a simple but important observation. Because $M$ is closed (i.e. has no boundary), compactness estimates for $G_{q}$ must hold at symmetric levels. For subelliptic estimates, this observation is in \cite{Kohn81}, Proposition on p.~255, and \cite{Koenig04}, p.~289. Koenig's construction works verbatim for compactness estimates. 
\begin{lemma}\label{symmetric}
Let $M$ be a smooth pseudoconvex compact CR-submanifold of $\mathbb{C}^n$ of hypersurface type (without boundary). Then, for $1\leq q\leq m-2$,
 $$G_q\;\text{ is compact}\; \Leftrightarrow\; G_{m-q-1}\;\text{ is compact}.$$
\end{lemma}
\begin{proof}
In order to verify that the compactness estimates \eqref{compest2} hold, it suffices, via  a partition of unity, to see that they hold for forms supported in a local coordinate chart. When derivatives hit the cutoff functions on the right-hand side of \eqref{compest2}, $\|u\|^{2}$ appears, but with $\varepsilon$ in front. These terms can therefore be absorbed into the left-hand side. So fix a local boundary chart. For forms supported in this chart, define a $(0,m-1-q)$-form
\begin{equation}\label{T-q}
T_q(\sideset{}{'}\sum_{\vert J\vert=q}\,u_J\overline{\omega_J}\,) \;\;\;\;=\sideset{}{'}\sum_{|J|=q, |K|=(m-1-q)
} \varepsilon_{(1,\dots, m-1)}^{JK} u_J\overline{\omega_K}\; ,
\end{equation}
where $\varepsilon_{(1,\dots, m-1)}^{JK}$ is the Kronecker symbol. Then, $T_{(m-1-q)}T_{q}u = (-1)^{q(m-1-q)}u$ (by computation), and the norms of $u$ and $T_{q}u$ are comparable. Moreover, (use \eqref{local}, \eqref{adjoint})
\begin{equation}\label{intertwine1}
\overline{\partial}_{M}T_{q}u = (-1)^{q}T_{(q-1)}(\overline{\partial}_{M}^{*}u) + \text{terms of order zero}\;,
\end{equation}
and
\begin{equation}\label{intertwine2}
\overline{\partial}_{M}^{*}T_{q}u = (-1)^{q+1}T_{(q+1)}(\overline{\partial}_{M}u) + \text{terms of order zero}\;.
\end{equation}
Details are in \cite{Koenig04}, p.~289. The intertwining properties \eqref{intertwine1} and \eqref{intertwine2} now make it clear that compactness estimates hold on $(0,q)$-forms supported in the fixed chart if and only if they hold on $(0,m-1-q)$-forms. By the remark at the outset, this completes the proof.
\end{proof}
The symmetry in Lemma \ref{symmetric} does not hold for compactness estimates for the $\overline{\partial}$-Neumann operator on a pseudoconvex domain. For example, on a convex domain, $N_{q}$ is compact if and only if the boundary contains no (germs of) $q$-dimensional complex varieties (\cite{FuStraube98, FuStraube99}). Therefore, if the boundary of the (convex) domain contains analytic discs, but no higher dimensional complex varieties, $N_{1}$ will fail to be compact, whereas $N_{(n-1)}$ will be compact. The argument from the proof of Lemma \ref{symmetric} breaks down: (the analogues of) the operators $T_{q}$ in general produce forms that fail to be in the domain of $\overline{\partial}^{*}$. On the other hand, compactness of the $\overline{\partial}$-Neumann operator percolates up the complex that is, if $N_{q}$ is compact, then so is $N_{q+1}$ (\cite{Straube10a}, Proposition 4.5), but compactness of the complex Green operator 
does 
not. If it did, compactness at one level would imply compactness at all levels, by Lemma \ref{symmetric}. However, it is known that on the boundary of a convex domain in $\mathbb{C}^{n}$, compactness of $G_{q}$ is equivalent to the absence of complex varieties of dimensions $q$ and $(n-1-q)$ from the boundary (\cite{RaichStraube08}, Theorem 1.5). So if the boundary of a convex domains contains an analytic disc, but no higher dimensional varieties, then $G_{2}, G_{3}, \hdots, G_{n-3}$ will be compact, but $G_{1}$ and $G_{n-2}$ will not be.

\smallskip

We now introduce a sufficient condition for compactness of $G_{q}$. In order to do so, we need a notion that was introduced into the literature in \cite{Catlin84b} (for $q=1$; see \cite{FuStraube99} for $q>1$). Let $K$ be a compact subset of $\mathbb{C}^{n}$. We say that \emph{$K$ satisfies property $(P_q)$} if for all $A>0$, there is a ${C}^2$ function $h_A$ defined on a neighborhood $U_A$ of $K$ such that $0\leq h_A\leq 1$ and 
\begin{equation}\label{propertyP}
\sideset{}{'}\sum_{\vert K\vert=q-1}\sum_{j,k=1}^n\dfrac{\partial^2h_A}{\partial z_j\partial\overline{z_k}}(z)w_{jK}\overline{w_{kK}}\geq A\vert w\vert^2\;,\;z\in U_{A}\;,\;w\in\Lambda^{0,q}_{z}\;.
\end{equation}
It will be useful, to have two reformulations of \eqref{propertyP}. First, 
\eqref{propertyP} is equivalent to saying that the sum of any  $q$-eigenvalues (equivalently, the smallest $q$) of the complex Hessian $\left(\partial^2h_A/\partial z_j\partial\overline{z_k}(z)\right)$ is at least equal to $A$. Note that this reformulation immediately gives $(P_q)\Rightarrow (P_{q+1})$ (because the $(q+1)$st eigenvalue necessarily has to be positive). Second, 
\begin{equation}\label{ortho}
\sum_{s=1}^{q}\sum_{j,k=1}^{n}\frac{\partial^{2}h_{A}(z)}{\partial z_{j}\partial\overline{z_{k}}}(t^{s})_{j}\overline{(t^{s})_{k}} \geq A\;, \text{whenever}\;t^{1}, \hdots, t^{q}\;\text{are orthonormal in}\;\mathbb{C}^{n}
\end{equation}
(see for example \cite{Straube10a}, Lemma 4.7, for these facts from multilinear algebra). A variant of property $(P_{q})$, $\widetilde{(P_{q})}$, was introduced in \cite{McNeal02} and shown to imply compactness of the $\overline{\partial}$-Neumann operator $N_{q}$. While it is easy to see that $(P_{q}) \Rightarrow \widetilde{(P_{q})}$, the exact relationship between the two is not understood. Similarly, whether $\widetilde{(P_{q})}$ can substitute for $(P_{q})$ also in the context of the Green operator does not seem to have been investigated.

The motivation in \cite{Catlin84b} was to have a sufficient condition on the boundary of a domain (i.e. $K=b\Omega$) for global regularity of the $\overline{\partial}$-Neumann operator (via compactness) that was verifiable on a reasonably large class of domains that included some domains that were not of finite type. Property $P_{1}$ was then investigated in \cite{Sibony87b}. In particular, Sibony realized that the notion fits beautifully into the context of Choquet theory for the cone of continuous functions on $K$ which can be approximated uniformly by functions that are plurisubharmonic in a neighborhood of $K$. For example, $K$ satisfying $(P_{1})$ is equivalent to this cone being all of $C(K)$; it is also equivalent to the Choquet boundary of this cone being all of $K$. Among other things, this analysis allowed Sibony to give examples of domains whose boundaries contain big (in the sense of having positive measure) sets of points of infinite type that still satisfy $(P_{1})$ (and thus have compact, 
hence globally regular, $\overline{\partial}$-Neumann operators). Note that if the above mentioned cone has to 
be all of $C(K)$, then the maximum principle 
for subharmonic functions immediately implies that $K$ cannot contain analytic discs. Thus property $(P_{1})$ should be thought of as excluding analytic structure from $K$ in a potential theoretic way. For the most part, this analysis carries over to $(P_{q})$, as observed in \cite{FuStraube99}. In particular, $(P_{q})$ should be thought of as excluding `$q$-dimensional analytic structure' from $K$; that is, $K$ must not contain $q$-dimensional complex varieties (or $q$-dimensional analytic polydiscs) (\cite{Straube10a}, Lemma 4.20). A thorough discussion of property $(P_{q})$ can be found in \cite{FuStraube99} and in sections 4.4 -- 4.10 of \cite{Straube10a}, in addition to the original reference \cite{Sibony87b} (for $q=1$).

While $(P_{q})$ is by now classical as a sufficient condition for compactness of the $\overline{\partial}$-Neumann operator, it has made its appearance in the context of the complex Green operator $G_{q}$ only relatively recently. Namely, first in \cite{RaichStraube08} when $M$ is the boundary of a smooth bounded pseudoconvex domain, then in \cite{Raich10} when $M$ is a smooth compact (closed) orientable pseudoconvex CR-submanifold of $\mathbb{C}^{n}$ of hypersurface type\footnote{Raich uses a condition called (CR-$P_{q})$, and he shows that $(P_{q}) \Rightarrow$ (CR-$P_{q})$. It turns out that the two conditions are actually equivalent for imbedded CR-manifolds (\cite{Straube10}). However, (CR-$P_{q})$ makes sense on CR-manifolds of hypersurface type that are not imbedded, and implies compactness of $G_{q}$ in this case also; this is shown in \cite{KhanhPintonZampieri12}.}, and finally in \cite{Straube10}, with a different proof that does not require $M$ to be orientable (see also \cite{
KhanhZampieri11} for a more general treatment of some of these results). 
These results have been generalized to CR-manifolds that are not assumed imbedded in \cite{KhanhPintonZampieri12}. Since our focus in this survey is on imbedded CR-manifolds, we only state the result for this case, for which we will sketch a proof that uses the simple geometric ideas from section \ref{CRsubmanifolds} to reduce to the case of an actual hypersurface. A nontrivial complication arises that reveals an intriguing issue with property $(P_{q})$ when $q>1$. Note that the symmetry in the assumptions (and the conclusion) is dictated by Lemma \ref{symmetric}.
\begin{theorem}\label{compact-q}
Let $M$ be a smooth compact (closed) pseudoconvex CR-submanifold of $\mathbb{C}^{n}$  of hypersurface type of CR-dimension $(m-1)$. Assume that $M$ satisfies $(P_q)$ and $(P_{m-1-q})$, (equivalently, $(P_{k})$, where $k=min\{q, m-1-q\}$), 
$1\leq q\leq m-2$. Then $G_q$ and $G_{m-1-q}$ are compact.
\end{theorem}
Note that verifying $(P_q)$ and $(P_{m-1-q})$ is equivalent to verify $(P_{k})$, where $k=min\{q, m-1-q\}$. One also notices that the endpoint cases $q=0$ and $q=(m-1)$ are excluded here; in particular, $m\geq 3$ (i.e. the CR-dimension of $M$ is at least two; see the discussion following the statement of Theorem 1 in \cite{Straube10}). The proof in \cite{RaichStraube08} follows the general strategy in \cite{Shaw85} (see also \cite{ChenShaw01}), whereas \cite{Raich10} further develops the microlocal analysis in \cite{Nicoara06} (\cite{KhanhPintonZampieri12} is also based on microlocal methods). We now sketch the proof in \cite{Straube10}.
\vspace{-0.06in}
\begin{proof}[Sketch of proof of Theorem \ref{compact-q}]
As noted in the proof of Lemma \ref{symmetric}, it suffices to show that compactness estimates hold for forms with `local' support. We thus fix a point $p\in M$ and consider forms with support in a neighborhood of $p$ that is small enough so that the setup in section \ref{CRsubmanifolds} holds. That is, near $p$, $M$ is graphed over a hypersurface $\pi(M)$ in $\mathbb{C}^{m}$, via \eqref{map}.

Now the idea is to obtain compactness estimates on $M$ (near $p$) from such estimates on $\pi(M)$, via the CR-equivalence provided by $\pi|_{M}$. There is a technical point here that one has to be mindful of: the tangential Cauchy--Riemann operators (defined extrinsically, as we did here) do not commute with the pullbacks under a CR-equivalence (see \cite{Boggess91}, section 9.2). Nevertheless, it is possible to set up a correspondence between $(0,q)$-forms on $M$ and on $\pi(M)$ so that $\overline{\partial}_{M}$ corresponds to $\overline{\partial}_{\pi(M)}$ \emph{modulo terms of order zero}; details are in \cite{Straube10}, p. 4113--4114. Just as in the proof of Lemma \ref{symmetric}, these terms are benign for compactness estimates (compare also Remark 6 in \cite{Straube10}). Therefore, to prove Theorem \ref{compact-q}, it suffices to establish compactness estimates on $\pi(M)$ for forms supported near $\pi(p)$. 

We would like to use \cite{RaichStraube08}, so we need $\pi(M)$ to be part of the boundary of a smooth bounded pseudoconvex domain. Denote by $\Omega$ such a domain, which is moreover strictly pseudoconvex at the points of $b\Omega\setminus \pi(M)$ (see \cite{Bell86} and the discussion in \cite{Straube10}, p.4113). Ideally, $\Omega$ would satisfy $(P_{q})$ and $(P_{m-1-q})$, and we would be done. This strategy works for $q=1$; when $q>1$, things turn out to be more complicated.

Because $M$ satisfies $(P_{q})$ and $(P_{m-1-q})$, it does not contain (germs of) complex manifolds of dimension $(m-1)$, so $\pi(M)$ does not either. As noted in section \ref{CRsubmanifolds}, this implies (Tr\'{e}preau's theorem, \cite{Trepreau06}) that the graphing functions $h_{1}, \hdots, h_{n-m}$ extend as holomorphic functions to one side of $\pi(M)$, near $\pi(p)$, 
as in \eqref{extend}\footnote{We deviate slightly from \cite{Straube10} here. At this step there, the mapping functions are extended merely as $C^{\infty}$ functions, but whose $\overline{\partial}$ vanishes to infinite order on $\pi(M)$; this is sufficient to obtain the desired functions $h_{A}^{\prime}$ 
below.}. Note that it is clear which side is the pseudoconvex side: because $\pi(M)$ is pseudoconvex and does not contain (germs of) complex $(m-1)$-dimensional manifolds near $\pi(p)$, it is not Levi flat in any neighborhood of $\pi(p)$, so that there are strictly positive eigenvalues of the Levi form arbitrarily close to $\pi(p)$.
Let now $q=1$, and let $K$ be a compact subset of $\pi(M)\cap b\Omega$. $(\pi|_{M})^{-1}(K)$ is a compact subset of $M$, and pulling back the (plurisubharmonic, since $q=1$) functions in the definition of property $(P_{1})$ for $M$ gives the following property of $K$. For all $A>0$, there 
exists a neighborhood $U_{A}$ of $K$, a plurisubharmonic function $h_{A}$ on $U_{A}\cap \Omega$, with $0\leq h_{A}\leq 1$ and with the smallest eigenvalue of its complex Hessian at least $A$. Because $h_{A}$ is also $C^{2}$ up to $b\Omega$, we can extend it across $b\Omega$ to obtain a neighborhood $U_{A}^{\prime}$ and a plurisubharmonic function $h_{A}^{\prime}$ on $U_{A}^{\prime}$ with $-1\leq h_{A}^{\prime}\leq 2$ and with the smallest eigenvalue of its complex Hessian at least equal to $A/2$. Replacing $h_{A}^{\prime}$ by $(1/3)(h_{A}^{\prime}+1)$ 
and rescaling the constant $A$ in the definition of property $(P_{1})$ for $K$ shows that $K$ does satisfy $(P_{1})$. As a result, $b\Omega$ is the countable union of compact sets, all of which satisfy $(P_{1})$ (since compact subsets of the strictly pseudoconvex part of the boundary automatically satisfy $(P_{1})$, see the proof of Corollary 4.16 in \cite{Straube10a}; this uses a local version of Theorem 2 in \cite{Catlin84b}). Such a union must itself satisfy $(P_{1})$ (\cite{Sibony87b}, Proposition 1.9, \cite{Straube10}, Corollary 4.14), and hence $(P_{q})$ for all $q$ (in 
particular for $q=(m-2)$). Therefore, by \cite{RaichStraube08}, Theorem 1.4, the complex Green operator on $(0,1)$-forms on $b\Omega$ is compact. By what we said above, this completes the proof of Theorem \ref{compact-q} for the case $q=1$.
 
When $q>1$, the above approach must be modified. We can no longer assert that the pullbacks $h_{A}$ from the previous paragraph have the property that the sum of the smallest $q$ eigenvalues of their Hessians is at least $A$. \eqref{ortho} above indicates that the lack of conformality of biholomorphic maps ($n$ is at least two) is the culprit. It is indeed not clear what happens to property $(P_{q})$ under biholomorphic maps when $q>1$. This is the `nontrivial complication' alluded to above. From the perspective of a sufficient condition for compactness, this situation is unsatisfactory; it is not hard to see, at least for the $\overline{\partial}$-Neumann operators, that compactness is invariant under biholomorphic maps. One would like the sufficient condition to exhibit the same invariance. A more complete discussion of this issue, and a possible remedy, is in Remarks 3 and 4 in \cite{Straube10}. The remedy amounts to changing the metric. This is also the approach in \cite{Straube10}. The Euclidean metric 
on vectors (and forms) on  $\Omega$ is modified near $\pi(M)$, so that the Hessians of the pullbacks $h_{A}$ satisfy \eqref{ortho} with respect to this modified metric (that is, for $t^{1}, \hdots, t^{q}$ orthonormal with respect to that metric). One then has to derive analogues, for this metric,  of the  weighted estimates (19) in \cite{RaichStraube08}. Once these are in hand, one can follow \cite{RaichStraube08} fairly closely to obtain compactness of $G_{q}$ and $G_{m-1-q}$ on $b\Omega$. The details require some work, for which we refer the reader to \cite{Straube10}.
\end{proof}
Property $(P_{q})$ is potential theoretic in nature. There are also known sufficient conditions for compactness that are geometric in nature, both for the $\overline{\partial}$-Neumann operator (\cite{Straube04, MunasingheStraube07}) and for the complex Green operator (\cite{MunasingheStraube12}). 

The idea underlying these conditions is most readily apparent in the case of comparable eigenvalues of the Levi form (in fact, the original paper \cite{Straube04} dealt with domains in $\mathbb{C}^{2}$). Namely, for such CR-submanifolds (of hypersurface type), the $L^{2}$-norms of complex tangential derivatives of a form $u$ are controlled by the $L^{2}$-norms of $\overline{\partial}_{M}u$, $\overline{\partial}_{M}^{*}u$, and $u$ (these estimates are often referred to as `maximal estimates'). Near a strictly pseudoconvex point, `everything is under control', by virtue of pseudolocal subelliptic estimates (cf. \cite{ChenShaw01}, Theorems 8.2.5 and 8.3.5, \cite{Kohn85}, Theorem 2.5). Therefore, if for arbitrarily short times, there is a complex tangential field $Z_{p,\varepsilon}$ along which a patch near a weakly pseudoconvex point $p$ can flow to a strictly pseudoconvex patch, the $L^{2}$-norm of $u$ on the weakly pseudoconvex patch is controlled by its norm on the strictly pseudoconvex patch (which is 
under 
control), plus an arbitrarily small 
time (i.e. $\varepsilon$) times $(\|\overline{\partial}_{M}u\|+\|\overline{\partial}_{M}^{*}u\|+\|u\|)$ (because the latter term dominates $\|Z_{p,\varepsilon}u\|$).  Modulo technicalities such as overlapping patches, this readily leads to the compactness estimates \eqref{compest2}.

When the comparable eigenvalues assumption is dropped, the scheme from the previous paragraph only allows to estimate the $L^{2}$-norm of the derivative of a certain microlocal portion of $u$ (usually referred to as the $+$ microlocalization). However, this turns out to be sufficient, because this estimate is obtained for all form levels, and the $T_{q}$ operators in \eqref{T-q} interchange the positive and negative microlocalizations of $u$ (the part usually referred to as the $0$-part is under control by ellipticity). There are additional conditions involved. One needs to impose some control on the fields $Z_{p,\varepsilon}$ vis-\`{a}-vis the smallest eigenvalue of the Levi form at each point. Moreover, in order to control the overlap of patches after they flow, one needs a uniform bound on the divergence of the fields $Z_{p,\varepsilon}$. This makes the most general formulation somewhat technical, and we refer the reader to \cite{MunasingheStraube12}, Theorem 1. Instead, we present two instances 
where all these conditions fit nicely into a simple geometric condition, namely a cone condition.

Denote by $K$ the set of weakly pseudoconvex points of $M$; $K$ is a compact subset of $M$. We say that \emph{$M\setminus K$ satisfies a complex tangential cone condition} if there is a (possibly small) cone $\mathcal{C}$ in $\mathbb{R}^{2n}\approx \mathbb{C}^{n}$ so that the following holds. For each $p\in K$, there exists a complex tangential direction so that when $\mathcal{C}$ is moved by a rigid motion to have vertex at $p$ and axis in the given complex tangential direction, then $\mathcal{C}\cap M$ is contained in $M\setminus K$. The phrase `$M\setminus K$ satisfies a cone condition with axis in the null space of the Levi form' shall have the obvious meaning.
\begin{proposition}\label{geometric}
Let $M$ be a smooth compact pseudoconvex CR-submanifold of $\mathbb{C}^{n}$ of hypersurface type (without boundary), of CR-dimension at least two (i.e. $m\geq 3$). Assume 

(i) The Levi form of $M$ has comparable eigenvalues at each point, and $M\setminus K$ satisfies a complex tangential cone condition.

or

(ii) The Levi form has only one degenerate eigenvalue at each point of $K$, and $M\setminus K$ satisfies a cone condition with axis in the null space of the Levi form (at each point).
\smallskip

\noindent Then,  $G_{q}$ is compact, $1\leq q\leq (m-2)$.
\end{proposition}

\begin{proof}. 
Case (i) is Corollary 1 in \cite{MunasingheStraube12} combined with Example 1 there. Case (ii) is in Example 4 in \cite{MunasingheStraube12}. 
\end{proof}
When $q=0$ or $q=(m-1)$, one still has compactness estimates under the usual orthogonality conditions, analogous to estimates (6) and (7) in Theorem 2 of \cite{MunasingheStraube12}; in this form, the result also holds when the CR-dimension of $M$ is one, i.e. $m=2$. 

It is worthwhile to note that the cases (i) and (ii) in Proposition \ref{geometric} are mutually exclusive (assuming that $K$ is not empty). Indeed, the comparable eigenvalues condition is equivalent to maximal estimates (locally, \cite{Derridj91b}, Th\'{e}or\`{e}me on p. 633 and references there), while (ii) implies that the Levi form is locally diagonalizable (\cite{MunasingheStraube07}, Example 3 and Remark 5; compare also \cite{Machedon88}, Lemma 2.1). Yet if $p\in M$ is a weakly pseudoconvex point near which maximal estimates hold, the Levi from cannot be diagonalizable near $p$ (\cite{Derridj91}). \cite{Derridj91b} deals with actual hypersurfaces, and \cite{Derridj91} deals with $\overline{\partial}$, rather than with $\overline{\partial}_{M}$, so that the previous argument needs justification. One option is to check that the arguments in these papers can be modified to work in our situation; we expect this to be the case, but we have not checked. Alternatively, we can take a more pedestrian route and 
invoke the idea from (the sketch of) the proof of Theorem \ref{compact-q} and work with a pseudoconvex domain $\Omega \subset \mathbb{C}^{m}$ that contains a piece of $\pi(M)$, near $\pi(p)$, in its boundary. Then the Levi form of $b\Omega$ at $\pi(p)$ will have exactly one degenerate eigenvalue if and and only if this is the case for the Levi form of $M$ at $p$. Similarly, it will have comparable eigenvalues if and only if this is the case for the Levi form of $M$ at $p$. Now we use \cite{Derridj91}, Th\'{e}or\`{e}me 2.2 to conclude that the latter happens if and only if $\overline{\partial}$ on $\Omega$ satisfies maximal estimates at $\pi(p)$, and \cite{Derridj91}, Th\'{e}or\`{e}me 7.1 to see that this contradicts diagonalizability of the Levi form of $b\Omega$ at $\pi(p)$, and hence the assumption that this Levi form only have one degenerate eigenvalue.

\smallskip

It is natural to ask what the relationship is between the potential theoretic conditions in Theorem \ref{compact-q} and the geometric conditions discussed above. Since the latter do not discriminate among form levels, one would for example ask how $(P_{1})$ relates to the flow conditions. This question is not understood at all. In our opinion, it is well worthwhile pursuing, with an eye toward characterizing compactness of the complex Green operators in terms of properties of $M$.

\section{Sobolev estimates}\label{Sobolev}

In this section, $M$ is again a smooth, compact, pseudoconvex, orientable CR-submanifold of hypersurface type of CR-dimension $(m-1)$. We address the question of when the complex Green operators on $M$ satisfy estimates in $L^{2}$-Sobolev spaces. We will also consider `the' Szeg\"{o} projection. More precisely, recall that $H_{q}$ denotes the orthogonal projection onto $\mathcal{H}_{q}(M)$ (so that $H_{0}$ is the usual Szeg\"{o} projection from $L^{2}(M)$ onto the square integrable CR-functions on $M$). In addition, we consider the  orthogonal projections $S^{\prime}_q:L^2_{(0,q)}(M)\rightarrow Im(\bar\partial_M)$, where $Im(\overline{\partial}_{M})$ is interpreted to be $\{0\}$ when $q=0$, and $S^{\prime\prime}_q: L^2_{(0,q)}(M)\rightarrow Im(\bar\partial^*_M)$, where $Im(\overline{\partial}_{M}^{*})$ is interpreted as $\{0\}$ when $q=(m-1)$. For $0\leq q\leq (m-1)$, these projections are associated with the following orthogonal decomposition: 
\begin{equation}\label{orth decomp}
u=S'_qu+S''_qu+H_qu\;,\; u\in L^{2}_{0,q}(M)\;.
\end{equation}
When proving Sobolev estimates for the complex Green operators $G_{q}$ and the related operators $S_{q}^{\prime}$ and $S_{q}^{\prime\prime}$, one needs control over certain components of commutators of vector fields (measuring derivatives) with $\overline{\partial}_{M}$ and $\overline{\partial}_{M}^{*}$. These components are conveniently expressed in terms of a $1$-form which we now introduce.

\medskip

Recall from section \ref{CRsubmanifolds} that since $M$ is orientable, there exists  a purely imaginary vector field $T$, orthogonal to $T^{\mathbb{C}}M$, and of unit length. Denote by $\eta$ the purely imaginary $1$-form on $M$ dual to $T$, that is, verifying $\eta(T)\equiv 1$ and $\eta(T^{1,0}M\oplus T^{0,1}M)\equiv 0$. The form we are interested in is
\begin{equation}\label{alpha}
 \alpha=-\mathcal{L}_T\eta\; ,
 \end{equation}
the negative Lie derivative of $\eta$ in the direction $T$. Note that $\alpha$ is real (since both $T$ and $\eta$ are purely imaginary). To see how $\alpha$ acts on $T^{0,1}(M)$, let $L \in T^{1,0}(M)$. Then
\begin{equation}\label{actionofalpha}
\alpha(\overline{L})  =-(\mathcal{L}_T\eta)(\overline{L}) = -\left(T\eta(\overline{L})-\eta([T,\overline{L}])\right)
= \eta([T,\overline{L}])\;
\end{equation}
(since $\eta(\overline{L})\equiv 0$). In other words, $\alpha(\overline{L})$ is the $T$-component of $[T,\overline{L}]$ modulo $T^{1,0}M\oplus T^{0,1}M$. Taking conjugates in \eqref{actionofalpha} shows that $\alpha$ acts in the same way on $T^{1,0}(M)$, hence on $T^{1,0}(M)\oplus T^{0,1}(M)$. Components of commutators as in \eqref{actionofalpha} (and their conjugates) are exactly what needs to be controlled when proving Sobolev estimates for the complex Green operators; herein lies the significance of the form $\alpha$ for this section. This form was introduced into the literature by D'Angelo \cite{D'Angelo80, D'Angelo86, D'Angelo93}, also in the context of commutators as above.

An important property of \emph{$\alpha$ is that it is closed on the null space of the Levi form}: 
\begin{lemma}\label{alphaclosed}
\begin{equation}\label{closed}
d\alpha|_{{\mathcal{N}_z}}=0,\quad \forall z\in M\;.
\end{equation}
\end{lemma}
Here, $\mathcal{N}_{z}$ denotes the null space of the Levi form of $M$ at $z\in M$. We refer the reader to \cite{BoasStraube93}, Lemma on page 230, for the proof. Pseudoconvexity of $M$ is essential.

Lemma \ref{alphaclosed} implies in particular that when $S\subset M$ is a complex submanifold of $M$, the restriction of $\alpha$ to $S$ defines a DeRham cohomology class on $S$ (\cite{BoasStraube93}). It is interesting that when $S\subset b\Omega$, where $\Omega$ is a smooth bounded pseudoconvex domain, the size of this class in a suitable norm on cohomology determines whether or not $\overline{\Omega}$ admits a Stein neighborhood basis; details are in \cite{BF78}.

\smallskip

We say that \emph{$\alpha$ is exact on the null space of the Levi form} if there exists $h\in C^\infty(M)$ such that 
\begin{equation}\label{exact}
dh(L)(z)=\alpha(L)(z), \quad L\in\mathcal{N}_z,\; z\in M\; .
\end{equation}
Exactness of $\alpha$ on the null space of the Levi form turns out to be precisely the condition one is led to when looking for vector fields that have good commutation properties with $\overline{\partial}_{M}$ and $\overline{\partial}_{M}^{*}$. Theorem \ref{Theo4}, from \cite{StraubeZeytuncu15}, is the main result of this section. It is not obvious how widely Theorem \ref{Theo4} applies, but we shall see that the closedness condition on $\alpha$ can indeed be verified on large classes of CR-submanifolds. 
Denote by $\|u\|_{s}$ the norm of $u$ in $W^{s}_{(0,q)}(M)$.
\begin{theorem}\label{Theo4}
Let $M$ be a smooth compact pseudoconvex orientable CR-submanifold of $\mathbb{C}^{n}$ of hypersurface type, of CR-dimension $(m-1)$. Assume that $\alpha = \alpha_{M}$ is exact on the null space of the Levi form. Then, for every nonnegative real number $s$, there is a constant $C_{s}>0$ such that for all $u\in L^{2}_{(0,q)}(M)$, $0\leq q\leq (m-1)$,
\begin{equation}\label{Theo4-1}
\Vert S'_qu\Vert^2_s+\Vert S''_qu\Vert^2_s+\Vert H_qu\Vert^2_s\leq C_s \Vert u\Vert^2_s\;,
\end{equation}
\begin{equation}\label{Theo4-2}
\Vert u\Vert^2_s \leq C_s \left(\Vert \bar\partial_Mu\Vert^2_s +\Vert \bar\partial^*_Mu\Vert^2_s+\Vert  u\Vert^2_{L^2_{(0,q)}(M)} \right),\quad 1\leq q\leq m-2 \;,
\end{equation}
\begin{equation}\label{Theo4-3}
\Vert u\Vert^2_s \leq C_s \left(\Vert \bar\partial_Mu\Vert^2_s +\Vert \bar\partial^*_Mu\Vert^2_s\right), \quad u\in\mathcal{H}_q(M)^{\perp},
\end{equation}
\begin{equation}\label{Theo4-4}
\Vert G_qu\Vert_s^2\leq C_s\Vert u\Vert_s^2 \;.
\end{equation}
\end{theorem}
\begin{proof}[Sketch of proof of Theorem \ref{Theo4}]
We mainly want to sketch how the condition that $\alpha = \alpha_{M}$ be closed on the null space of the Levi form arises when looking for the kinds of vector fields needed to run an adaption of the machine developed in \cite{BoasStraube91}\footnote{There is more to the story than meets the eye; closedness of $\alpha$ on the null space of the Levi form and existence of `suitable' vector fields are `almost' equivalent, and are `almost' equivalent to two other conditions; see \cite{StraubeSucheston02}, where these notions are made precise.}; this amounts to Proposition 1 in \cite{StraubeZeytuncu15}. The details of what adaptions are necessary, and how to run the machine, are quite technical and lengthy, and we refer the reader to \cite{StraubeZeytuncu15}.
 
The vector fields needed to prove the desired Sobolev estimates are as follows (\cite{StraubeZeytuncu15}, Theorem 4). There should exist a family $\{X_{\varepsilon}\}$, $0<\varepsilon<\varepsilon_{0}$ for some $\varepsilon_{0}$, of real vector fields on $M$ with the property
\begin{equation}\label{vectorfields}
 1/C\leq |(X_{\varepsilon})_{T}|\leq C\;,\;\text{and}\;\left|\left([X_{\varepsilon}, \overline{Z}]\right)_{T}\right| \leq \varepsilon \;,\;Z\in T^{1,0}(M)\;,\;|Z|=1\;,
\end{equation}
where $C$ is a constant independent of $\varepsilon$ and the subscript $T$ denotes the $T$-component of a vector field modulo $T^{1,0}(M)\oplus T^{0,1}(M)$. The first condition says that the fields $X_{\varepsilon}$ are uniformly transversal to $T^{\mathbb{C}}(M)$. Two facts are essential for the construction of the family $\{X_{\varepsilon}\}$. First, complex tangential derivatives are benign, so that only the $T$-components of the commutators have to be controlled (\cite{BoasStraube91}, Lemma 1, \cite{StraubeZeytuncu15}, inequality (35); compare also \cite{BoasStraube91a}, inequality (3)). Second, control over ($T$-components of) commutators with fields in strongly pseudoconvex directions can be established for free (\cite{StraubeZeytuncu15}, page 1087; compare also \cite{BoasStraube91a}, \cite{Straube10a}, section 5.7; and see below). These two facts are reflected in the conditions on $\alpha$: $\alpha$ only measures $T$-components of commutators (\eqref{actionofalpha}), and closedness is only required 
on the null space of the Levi form.

Locally, near a point $p$, choose a basis $L_{1}, \hdots, L_{(m-1)}$ of $T^{1,0}(M)$. Then, the fields $X_{\varepsilon}$ we are looking for can locally be written as
\begin{equation}\label{localfields}
 X_{\varepsilon} = e^{g_{\varepsilon}}T + \sum_{j=1}^{m-1}\left(b_{\varepsilon,j}L_{j}+\overline{b_{\varepsilon, j}L_{j}}\right)\;,
\end{equation}
for smooth functions $g_{\varepsilon}$ and $b_{\varepsilon, j}$ that are to be determined. We may choose the basis $\{L_{j}\}_{j=1}^{m-1}$ so that at $p$, it diagonalizes the Levi form (but not necessarily in a neighborhood). The field $Z$ in \eqref{vectorfields} is a linear combination of $L_{1}, \hdots, L_{(m-1)}$, whence we look at commutators $[X_{\varepsilon}, \overline{L_{k}}]$, $1\leq k\leq (m-1)$. A computation using \eqref{actionofalpha}, \eqref{localfields} and the fact that $L_{1}, \hdots, L_{(m-1)}$ diagonalizes the Levi form at $p$ gives, at the point $p$,
\begin{equation}\label{atP}
 \left[X_{\varepsilon}, \overline{L_{k}}\right]_{T}(p) = -e^{g_{\varepsilon}(p)}\left(\overline{L_{k}}g_{\varepsilon}(p)-\alpha(\overline{L_{k}})(p)\right) + b_{\varepsilon, k}(p)\left[L_{k}, \overline{L_{k}}\right]_{T}(p)\;
\end{equation}
(see \cite{StraubeZeytuncu15}, proof of Proposition 1, for details). First, consider the case where $L_{k}(p)\in \mathcal{N}_{p}$. Then the second term on the right-hand side of \eqref{atP} vanishes, and we are left with the first term only. A moment's contemplation shows how the condition on $\alpha$ now enters. If we set $g_{\varepsilon} = h$ for all $\varepsilon$, where $h\in C^{\infty}(M)$ is such that $dh|_{\mathcal{N}_{z}} = \alpha|_{\mathcal{N}_{z}}$, $z\in M$, then the first term also vanishes. Therefore, the commutator condition in \eqref{vectorfields} is satisfied at $p$: $\left|[X_{\varepsilon}, \overline{L_{k}}]_{T}\right|(p)=0$ (regardless of what $b_{\varepsilon, k}$ is, we will set it equal to zero below). On the other hand, for those $k$ where $L_{k}(p) \notin \mathcal{N}_{p}$, $[L_{k}, \overline{L_{k}}]_{T} \neq 0$, and we can solve for $b_{\varepsilon, k}(p)$ to make the right-hand side of \eqref{atP} equal to zero (with $h(p)$ in place of $g_{\varepsilon}(p)$). The latter is the meaning of 
controlling commutators with fields in strictly pseudoconvex directions for free, alluded to above.

The conclusion is that if we define $X_{\varepsilon, p}$ near $p$ by $e^{h}T+\sum_{L_{k}\notin \mathcal{N}_{p}}\left(b_{\varepsilon, k}L_{k}+\overline{b_{\varepsilon, k}L_{k}}\right)$, with the constants $b_{\varepsilon, k}$ from the previous paragraph, then $[X_{\varepsilon, p}, \overline{L_{k}}]_{T}(p) = 0$. By continuity, there is a neighborhood of $p$ such that 
$[X_{\varepsilon, p}, \overline{Z}]_{T}(z) < \varepsilon$ for $z$ in this neighborhood and $Z$ as in \eqref{vectorfields}. Covering $M$ with finitely many of these neighborhoods and patching together the locally defined fields via a partition of unity produces the desired family $\{X_{\varepsilon}\}$ (because the $T$-component of these local fields is defined globally, namely $e^{h}T$, the terms arising from derivatives hitting the cutoff functions have vanishing $T$-component, and so are harmless). For details, we refer again to \cite{StraubeZeytuncu15}, proof of Proposition 1.
\end{proof}
We have noted in section \ref{L-2} that while $\mathcal{H}_{q}(M)$ is finite dimensional for $1\leq q\leq (m-2)$, it need not be trivial. The following immediate consequence of \eqref{Theo4-2} is thus worth noting: harmonic forms are smooth.
\begin{corollary}\label{harmsmooth}
 Under the assumptions of Theorem \ref{Theo4}, $\mathcal{H}_{q}(M) \subset C^{\infty}_{(0,q)}(M)\;,\;1\leq q\leq (m-2)\;.$
\end{corollary}
The estimates in Theorem \ref{Theo4} are not independent. For example, in \cite{HPR13}, the authors show among other things that for $1\leq q\leq (m-2)$, regularity of $G_{q}$ is equivalent to that of the three  projections $S_{(q-1)}^{\prime}+H_{(q-1)}$, $S_{q}^{\prime}+H_{q}$, and $S_{(q+1)}^{\prime}$ are (the first two are the projections onto the respective kernels of $\overline{\partial}_{M}$). Analogous results for the $\overline{\partial}$-Neumann operators and the Bergman projections had been known for quite some time (\cite{BoasStraube90}).

\smallskip

A first application of Theorem \ref{Theo4} concerns the situation when $M$ is strictly pseudoconvex except for a complex submanifold, say $S$. As mentioned above, in view of Lemma \ref{alphaclosed}, the restriction of $\alpha_{M}$ to $S$ is closed and so defines a DeRham cohomology class on $S$ which we denote by $[\alpha|_{S}]$. The following corollary is an analogue of the main result in \cite{BoasStraube93}, which deals with the $\overline{\partial}$-Neumann problem on pseudoconvex domains. The more general assumption on the submanifold $S$ used there, namely that at each point of $S$ its (real) tangent space be contained in the null space of the Levi form of $M$, could also be used in the present context. But the special case formulated here is perhaps the most important case.
\begin{corollary}\label{submanifold}
Assume $M$ is as in Theorem \ref{Theo4}, strictly pseudoconvex except for a smooth complex submanifold $S$ of $M$ (smooth as a submanifold with boundary). If $[\alpha|_{S}] = 0$, then $\alpha$ is exact on the null space of the Levi form, and the conclusions of Theorem \ref{Theo4} and Corollary \ref{harmsmooth} hold.
\end{corollary}
The assumption on $[\alpha|_{S}]$ is of course satisfied when $S$ is simply connected. However, $[\alpha|_{S}]$ can vanish also when $S$ is not simply connected. Simple examples are annuli in the boundaries of certain Hartogs domains in $\mathbb{C}^{2}$ called `nowhere wormlike' in \cite{BoasStraube92} (considered there in connection with regularity of the Bergman projection). 
\begin{proof}[Proof of Corollary \ref{submanifold}]
Because $[\alpha|_{S}]=0$, there is a smooth function $\tilde{h}$ on $S$ (smooth up to the boundary of $S$) which satisfies $d_{S}\tilde{h}=\alpha|_{S}$. Now one extends $\tilde{h}$ to a smooth function $h$ on $M$ such that $d_{M}h=\alpha$ at all points of $S$. Details are in \cite{StraubeZeytuncu15}, section 4; compare also \cite{BoasStraube93}, section 3,  and \cite{Straube10a}, section 5.10.
\end{proof}

It would be very interesting to know whether a `converse' of Corollary \ref{submanifold} holds. That is, assume $M$ and $S$ are is as in Corollary \ref{submanifold}, and $G_{1}$ satisfies exact Sobolev estimates ($\|G_{1}u\|_{k} \leq C_{k}\|u\|_{k}$, $k\in \mathbb{N}$), or is just globally regular ($G_{1}(C^{\infty}_{(0,1)}(M)) \subset C^{\infty}_{(0,1)}(M)$). Does it follow that $[\alpha|_{S}] = 0\,$? This question is also open for the $\overline{\partial}$-Neumann operator when $M$ is the boundary of a pseudoconvex domain. In the case of the worm domains, where $N_{1}$ is known not to be globally regular (\cite{Christ96}), this failure of regularity is linked to the nonvanishing of the cohomology class of $\alpha$ on the critical annulus, compare \cite{Christ96}, page 1175 and \cite{BoasStraube92}, Proposition 1. A further result pointing in the same direction is in \cite{Barrett94}.

The ideas around Corollary \ref{submanifold} can also be developed when there is a foliation by complex manifolds in $M$, that is, $M$ is strictly pseudoconvex except for a flat piece that is foliated by complex submanifolds. Doing so reveals interesting connections with foliation theory. For example, in this situation, exactness of $\alpha$ on the null space of the Levi form amounts to solving $dh|_{\mathcal{L}}=\alpha|_{\mathcal{L}}$ for each leaf $\mathcal{L}$ of the foliation. This question turns out to be equivalent to whether or not the foliation can be defined \emph{globally} by a closed one-form, a question much studied in foliation theory. And the cohomology class of $\alpha$ on each leaf coincides with the infinitesimal holonomy of the leaf. We refer the reader to \cite{StraubeSucheston03}, \cite{BDD07}, section 3.6, and \cite{Straube10a}, section 5.11 for more on these connections.

\smallskip

The second application of Theorem \ref{Theo4} that we wish to discuss concerns submanifolds $M$ that can be defined by plurisubharmonic defining functions. Deducing the result from Theorem \ref{Theo4} relies on Theorem \ref{theo1}. When $M$ is the boundary of a pseudoconvex domain, Theorem 1 is not an issue, and the result was obtained in \cite{BoasStraube91}.

Because $M$ is orientable, there is a tubular neighborhood $V$ of $M$ so that in $V$, $M$ is defined by a set of $l=2n-(2m-1)$ smooth functions $\rho_{1}, \hdots, \rho_{l}$ whose gradients do not vanish simultaneously on $M$; $M=\{z\in V|\rho_{1}(z)=\hdots=\rho_{l}(z)=0\}$. The following result was recently shown in \cite{StraubeZeytuncu15}.
\begin{corollary}\label{plurisubdefining}  
 Let $M$ be a smooth compact pseudoconvex orientable CR-submanifold of $\mathbb{C}^{n}$ of hypersurface type. Assume that $M$ admits a set of plurisubharmonic defining functions in some neighborhood. Then $\alpha =\alpha_{M}$ is exact on the null space of the Levi form. Consequently, the conclusions of Theorem \ref{Theo4} and Corollary \ref{harmsmooth} hold.
\end{corollary}
The geometric observation needed is as follows (\cite{StraubeZeytuncu15}). Recall the one-sided complexification $\widehat{M}$ of $M$ from Theorem \ref{theo1}. Denote by $\widetilde{T}$ the real unit normal to $M$ pointing outside $\widehat{M}$, $\widetilde{T}=-iJT$. For $z\in M$, denote by $C_{z}$ the positive cone generated by the gradients of the $\rho_{j}$s in $\mathbb{C}^{n} \approx \mathbb{R}^{2n}$, $C_{z}=\{c_{1}\nabla\rho_{1}(z)+ \hdots +c_{l}\nabla\rho_{l}(z)|\,c_{1}\geq 0, \hdots, c_{l}\geq 0\}$, and by $\widehat{C_{z}}$ its dual cone.
\begin{lemma}\label{geometric2}
With the notations from above, let $c_{1}, \hdots, c_{l}$ be strictly positive constants. Then $\widetilde{T}\left(\sum_{j=1}^{l}c_{j}\rho_{j}\right) > 0$ on $M$. In other words, near $M$, $\widehat{M}$ is contained in the union $\cup_{z\in M}(z+\widehat{C_{z}})$.
\end{lemma}
\begin{proof}
Fix $z\in M$. Because $\widetilde{T}$ is not tangential to $M$, $\widetilde{T}\rho_{j}(z)\neq 0$ for at least one $j$. On the other hand, because the $\rho_{j}$s are plurisubharmonic, and $\widehat{M}$ is contained in the plurisubharmonic hull of $M$ (with respect to the functions plurisubharmonic in the neighborhood $V$ where the $\rho_{j}$ live), we have $\rho_{j}\leq 0$ on $\widehat{M}$ near $M$, so that $\widetilde{T}\rho_{j}(z)\geq 0$, $1\leq j\leq l$. Combining these two facts gives $\widetilde{T}(\sum_{j=1}^{l}c_{j}\rho_{j})(z) = \sum_{j=1}^{l}c_{j}\widetilde{T}\rho_{j}(z) > 0$. For the second statement, notice that this inequality implies that $-\widetilde{T}(z)\in \widehat{C_{z}}$. Because $-\widetilde{T}(z)$ is the interior normal to $M$ relative to $\widehat{M}$, the statement follows.
\end{proof}
\begin{proof}[Proof of Corollary \ref{plurisubdefining}(Sketch)]
Once $\widehat{M}$ and Lemma \ref{geometric2} are in hand, the proof that $\alpha$ is exact on the null space of the Levi form follows the argument when $M$ is the boundary of a domain (\cite{BoasStraube91a, BoasStraube91}, \cite{Straube10a}, sections 5.8, 5.9), but with some rephrasing. We indicate the steps below, and show in particular how the plurisubharmonicity of the defining functions enters into the argument.

Set $\rho:=\rho_{1}+\hdots+\rho_{l}$; $\rho$ can serve as a (one-sided) defining function for $M$ on $\widehat{M}$, and $(\partial\rho - \overline{\partial}\rho)$  is a nonvanishing multiple of $\eta$. Indeed, if $J^{*}$ denotes the adjoint of $J$ with respect to the pairing between vector fields and forms (see \cite{Boggess91}, page 42), we note that $J^{*}\partial = i\partial$ and $J^{*}\overline{\partial} = -i\overline{\partial}$. Then, we compute:
\begin{multline}\label{comp1}
\;\;\;\;\;\left(\partial\rho - \overline{\partial}\rho\right)(T) = -\left(\partial\rho - \overline{\partial}\rho\right)(J^{2}T) \\= -\left(J^{*}\partial\rho - J^{*}\overline{\partial}\rho\right)(JT) = -id\rho(JT) = d\rho(\widetilde{T}) > 0\;,\;\;\;\;\;
\end{multline}
in view of Lemma \ref{geometric2} (with $c_{1}=\hdots=c_{l}=1$). Noticing that $\left(\partial\rho - \overline{\partial}\rho\right)$ annihilates $T^{1,0}(M)\oplus T^{0,1}(M)$, and $\eta(T)\equiv 1$, we conclude from \eqref{comp1} that on $M$
\begin{equation}\label{def-h}
\left(\partial\rho - \overline{\partial}\rho\right) = e^{-g}\eta \;, \text{for some}\;g\in C^{\infty}(M)\;. 
\end{equation}
The argument in \cite{BoasStraube91} is formulated in terms of vector fields with the properties in \eqref{vectorfields}. In particular, when the defining function $\rho$ is plurisubharmonic, the family of vector fields are given by the (purely imaginary) field dual to $(\partial\rho-\overline{\partial}\rho)$, mod $T^{1,0}(M)\oplus T^{0,1}(M)$ (only the complex tangential parts depend on $\varepsilon$). This field equals $e^{g}T$, in view of \eqref{def-h}. Combining this with the discussion in the proof of Theorem \ref{Theo4} (see in particular \eqref{localfields}, \eqref{atP}) reveals that the
function $h:=g$ is the function we look for; that is, we claim that with this definition of $h$, $dh|_{\mathcal{N}_{z}} = \alpha|_{\mathcal{N}_{z}}$, $z\in M$. Compare also sections 5.8 and 5.9 in \cite{Straube10a}; a detailed argument is in \cite{StraubeSucheston02}, equivalence of (i) and (iv) in the Theorem on page 452. It is convenient to check directly that $h$ so defined has the required properties.

A computation (see \cite{StraubeZeytuncu15}, page 1084) shows that                                                                                                                                                                                                  
\begin{equation}\label{comp2}
 \alpha(\overline{L})(z) = dh(\overline{L})(z) + 2e^{h}i\partial\overline{\partial}\rho(-iT, \overline{L})(z)\;,\;z\in M\;,\;L\in\mathcal{N}_{z}\;.
\end{equation}
Both $iT$ and $J(iT)$ are tangent to $\widehat{M}$; therefore so is the $(1,0)$-part $(iT)_{(1,0)}$ of $iT$, and  $i\partial\overline{\partial}\rho(-iT, \overline{L})(z) = \partial\overline{\partial}\rho\left((-iT)_{(1,0)}, \overline{L}\right)(z)$. Because $\rho$ is plurisubharmonic (since the $\rho_{j}$ are, $1\leq j\leq l$), and therefore its restriction to $\widehat{M}$ is also, $i\partial\overline{\partial}\rho$ is positive semi-definite. Thus, when $L\in \mathcal{N}_{z}$, $\partial\overline{\partial}\rho\left((-iT)_{(1,0)}, \overline{L}\right)(z)=0$ (by Cauchy--Schwarz), and \eqref{comp2} gives
\begin{equation}\label{comp3}
  \alpha(\overline{L})(z) = dh(\overline{L})(z)\;,\;z\in M\;,\;L\in\mathcal{N}_{z}\;,
\end{equation}
as claimed.
\end{proof}

In view of Lemma \ref{symmetric}, the question arises whether a similar symmetry in form levels holds for Sobolev estimates for the complex Green operators. The $T_{q}$ operators from the proof of the lemma do not give this same symmetry, because the error terms of order zero that appear in the intertwining equations \eqref{intertwine1} and \eqref{intertwine2} are now no longer negligible. 

When $M$ is the boundary of a pseudoconvex domain in $\mathbb{C}^{n}$, $T_{q}$ operators that intertwine $\overline{\partial}_{M}$ and $\overline{\partial}_{M}^{*}$ without error terms have been constructed in the appendix in \cite{RaichStraube08}. It then follows that in this case, Sobolev estimates do indeed hold at symmetric levels. This question is not relevant for the results presented in this section, because the sufficient conditions for regularity in these results do not depend on the form level $q$. However, in the case of the $\overline{\partial}$-Neumann operator, the condition that the domain admit a plurisubharmonic 
defining function can be relaxed when the form level $q$ is greater than one: the complex Hessian should have the property that the sum of any $q$ eigenvalues (equivalently, the smallest $q$) be non negative (\cite{HerbigMcNeal05}, \cite{Straube10a}, Corollary 5.25). Thus to prove an analogue of this result for the complex Green operator $G_{q}$ on the boundary of a domain in $\mathbb{C}^{n}$, one should expect the sufficient condition to be that the domain admit a defining function whose complex Hessian has the property that the sum of the smallest $l$ eigenvalues is non negative, where $l$ is the smaller of $q$ and $n-1-q$. When $q>(n-1)/2$, this assumption is stronger than assuming only that the sum of the smallest $q$ eigenvalues is non negative. In this regard, the situation would thus be analogous to that in Theorem \ref{compact-q}. A similar discussion applies with the general sufficient condition for global regularity of the $\overline{\partial}$-Neumann operator $N_{q}$ given in 
\cite{Straube08} (see also \cite{Straube10a}, Theorem 5.22).

The construction in \cite{RaichStraube08} seems to rely very much on the fact that the boundary of a domain is an actual hypersurface in $\mathbb{C}^{n}$. It is not clear whether $\widehat{M}$ can substitute for $\mathbb{C}^{n}$ in the general case, where $M$ is only assumed to be of hypersurface type.

\bigskip
\bigskip
\emph{Acknowledgement:} The authors are grateful to the referee for thoughtful remarks that led to various improvements in the presentation.

\bigskip
\providecommand{\bysame}{\leavevmode\hbox to3em{\hrulefill}\thinspace}


\begin{thebibliography}{10}


\bibitem{Baracco1}
Baracco, Luca, The range of the tangential Cauchy--Riemann system to a CR embedded
manifold, \emph{Invent. Math.} \textbf{190}, (2012), 505--510.

\bibitem{Baracco2}
\bysame, Erratum to: The range of the tangential Cauchy--Riemann system to a CR
embedded manifold, \emph{Invent. Math.} \textbf{190}, (2012), 511--512.

\bibitem{Baracco3}
\bysame, Boundaries of analytic varieties, preprint, arXiv:1211.0787.

\bibitem{BDD07}
Barletta,~E., Dragomir,~S., and Duggal,~K.~L., \emph{Foliations in Cauchy-Riemann Geometry}, Mathematical Surveys and Monographs 140, Amer. Math. Soc., 2007.

\bibitem{Barrett94}
Barrett, David E., The Bergman projection on sectorial domains. In \emph{Operator Theory
for Complex and Hypercomplex Analysis}, \emph{Contemp. Math.} \textbf{212}, Amer. Math.
Soc., Providence, RI, 1998, 1--14.

\bibitem{BER99}
M.~Salah Baouendi, Peter Ebenfelt, and Linda Preiss Rothschild, \emph{Real Submanifolds
in Complex Space and Their Mappings}, Princeton University Press, Princeton, 1999.

\bibitem{BF78}
Bedford, Eric and Forn\ae ss, John Erik, Domains with pseudoconvex neighborhood systems, \emph{Invent. Math.} \textbf{47} (1978), 1--27.

\bibitem{Bell86}
Bell, S., Differentiability of the Bergman kernel and pseudolocal estimates, \emph{Math. Z.} \textbf{192}  (1986), 467--472.

\bibitem{BerhanuMendoza97}
Berhanu, S. and Mendoza, G.~A., Orbits and global unique continuation for systems of vector fields, \emph{J. Geom. Anal.} \textbf{7}, no. 2 (1997), 173--194.

\bibitem{BoasShaw86}
Boas, Harold P. and Shaw, Mei-Chi, Sobolev estimates for the Lewy operator on weakly
pseudoconvex
boundaries, \emph{Math. Ann.} \textbf{274}, no.2  (1986), 221--231.

\bibitem{BoasStraube90}
Boas, Harold P. and Straube, Emil J., Equivalence of regularity for the Bergman projection
and the $\overline{\partial}$-Neumann operator, \emph{Manuscripta Math.} \textbf{67}
(1990), 25--33.

\bibitem{BoasStraube91a}
\bysame, Sobolev estimates for the
$\overline\partial$-{N}eumann operator on domains in $\mathbb{C}^{n}$ admitting a defining
function that is plurisubharmonic on the boundary, \emph{Math. Z.}, \textbf{206} (1991),
81--88.

\bibitem{BoasStraube91}
\bysame, Sobolev estimates for the complex Green operator on
a class of weakly pseudoconvex boundaries, \emph{Commun. Partial Diff. Equations}
\textbf{16}, no.10 (1991), 1573--1582.

\bibitem{BoasStraube92}
\bysame, The Bergman projection on Hartogs domains in
$\mathbb{C}^{2}$, \emph{Trans. Amer. Math. Soc.} \textbf{331} (1992),
529--540.

\bibitem{BoasStraube93}
\bysame, De {R}ham cohomology of manifolds containing the points
of infinite type, and {S}obolev estimates for the
$\overline\partial$-{N}eumann problem, \emph{J. Geom. Anal.} \textbf{3}, Nr.3
(1993), 225--235.

\bibitem{Boggess91}
Boggess, Albert, \emph{CR Manifolds and the Tangential Cauchy-Riemann Complex}, Studies in Advanced Mathematics, CRC Press 1991.

\bibitem{Brink02}
Brinkschulte, J., Laufer's vanishing theorem for embedded CR manifolds, \emph{Math.~Z.} \textbf{239} (2002), 863--866.

\bibitem{Catlin84b}
Catlin, D., Global regularity of the {$\bar \partial
$}-{N}eumann problem. In \emph{Complex Analysis of Several Variables} (ed. by
Y.-T.~Siu). Proc. Sympos. Pure Math. 41, Amer. Math.
Soc., Providence 1984, 39--49.


\bibitem{ChenShaw01}
Chen, So-Chin and Shaw, Mei-Chi, \emph{Partial Differential Equations in
Several Complex Variables}, Studies in Advanced Mathematics 19, Amer. Math. Soc./International Press, 2001.

\bibitem{Christ96}
Christ, Michael, Global $C^{\infty}$ irregularity of the
$\overline{\partial}$-Neumann problem for worm domains, \emph{J. Amer.
Math. Soc.} \textbf{9}, Nr. 4 (1996), 1171--1185.

\bibitem{D'Angelo80}
D'Angelo, John, Finite type conditions for real hypersurfaces, \emph{J.~Diff.~Geometry} \textbf{14} (1980), 59--66.

\bibitem{D'Angelo86}
\bysame, Iterated commutators and derivatives of the Levi form, in \emph{Complex Analysis} (Univ. Park, PA, 1986), Lecture Notes in Mathematics 1268, 103--110.

\bibitem{D'Angelo93}
\bysame, \emph{Several Complex Variables and the Geometry of
Real Hypersurfaces}, Studies in Advanced Mathematics, CRC Press, Boca
Raton, 1993.



\bibitem{Derridj91}
Derridj, M., Domaines \`{a} estimation maximale, \emph{Math. Z.} \textbf{208} (1991), 71--88.

\bibitem{Derridj91b}
\bysame, Microlocalisation et estimations pour $\overline{\partial}_{b}$ dans quelques hypersurfaces pseudoconvexes, \emph{Invent. Math.} \textbf{104} (1991), no. 3, 631–-642. 


\bibitem{Fefferman00}
Fefferman, C., Editor's note on papers by Harvey--Lawson and Luk--Yau, \emph{Ann. Math.} \textbf{151} (2000), 875.


\bibitem{FuStraube98}
Fu, Siqi and Straube, Emil J., Compactness of the
{$\overline\partial$}-{N}eumann problem on convex domains, \emph{J. Funct.
Anal.} \textbf{159} (1998), 629--641.

\bibitem{FuStraube99}
\bysame, Compactness in the
{$\overline\partial$}-{N}eumann problem. In \emph{Complex Analysis and
Geometry} (ed. by J.~McNeal). Ohio State Univ. Math. Res. Inst.
Publ. 9, de Gruyter, Berlin 2001, 141--160.


\bibitem{HangesTreves83}
 Hanges, Nicholas and Tr\`{e}ves, Fran\c{c}ois, Propagation of holomorphic extendability
of CR functions, \emph{Math. Ann.} \textbf{263}, no.2 (1983), 157--177.

\bibitem{HPR13}
Harrington, Phillip S., Peloso Marco M., and Raich, Andrew S., Regularity equivalence of
the Szeg\"{o} projection and the complex Green operator, \emph{Proc. Amer. Math. Soc.} \textbf{143},
no. 1 (2015), 353–-367.

\bibitem{HarringtonRaich10}
Harrington, Phillip S. and Raich, Andrew, Regularity results for
$\overline{\partial}_{b}$
on CR-manifolds of hypersurface type, \emph{Commun. Partial Diff. Equations} \textbf{36},
no.1 (2011), 134--161.

\bibitem{HarveyLawson75}
Harvey, F. Reese and Lawson, H. Blaine, Jr., On boundaries of complex analytic varieties
I, \emph{Ann. of Math.} \textbf{102} (1975), 223--290.

\bibitem{HarveyLawson00}
\bysame, Addendum to Theorem 10.4 in ``Boundaries of analytic varieties'', arXiv:math/0002195v1.

\bibitem{HerbigMcNeal05}
Herbig, Anne-Katrin and McNeal, Jeffery D., Regularity of the Bergman projection on forms and plurisubharmonicity conditions,  \emph{Math. Ann.}  \textbf{336}  (2006), 335--359.

\bibitem{Hormander65}
H\"{o}rmander, L., $L^{2}$ estimates and existence theorems for the $\overline{\partial}$
operator, \emph{Acta Math.} \textbf{113} (1965), 89--152.


\bibitem{KhanhZampieri11}
Khanh, Tran V. and Zampieri, Giuseppe, Estimates for regularity of the tangential $\overline{\partial}$--system, \emph{Math. Nachr.} \textbf{284}, no. 17-18 (2011), 2212–-2224. 

\bibitem{KhanhPintonZampieri12}
Khanh, Tran V., Pinton, Stefano, and Zampieri, Giuseppe, Compactness estimates for $\Box_{b}$ on a CR manifold, \emph{Proc. Amer. Math. Soc.} \textbf{140}, no.~9 (2012), 3229–-3236.

\bibitem{Kneser36}
Kneser, H., Die Randwerte einer analytischen Funktion zweier Ver\"{a}nderlichen,
\emph{Monatshefte f. Math. u. Phys.} \textbf{43} (1936), 364--380.

\bibitem{Koenig04}
Koenig, K.~D., A parametrix for the $\overline{\partial}$-Neumann problem on pseudoconvex
domains of finite type, \emph{J. Func. Anal.} \textbf{216} (2004), 243--302.



\bibitem{Kohn85}
Kohn, J.~J., Estimates for $\overline{\partial}_{b}$ on pseudoconvex CR manifolds,
Pseudodifferential Operators and Applications (Notre Dame, Ind., 1984),
\emph{Proc. Sympos. Pure Math.} \textbf{43}, Amer. Math. Soc., Providence, RI (1985),
207--217.

\bibitem{Kohn81}
\bysame, Boundary regularity of $\overline{\partial}$, \emph{Recent Developments in Several Complex Variables} (Proc. Conf., Princeton Univ., Princeton, N. J., 1979), pp. 243–-260, Ann. of Math. Stud. \textbf{100}, Princeton Univ. Press, Princeton, N.J., 1981.

\bibitem{Kohn86}
\bysame, The range of the tangential Cauchy--Riemann operator, \emph{Duke Math. J.}
\textbf{53}, no.2  (1986), 525--545.




\bibitem{LukYau98}
Luk, Hing Sun and Yau, Stephen S.--T., Counterexample to boundary regularity of a strongly pseudoconvex CR submanifold: an addendum to a paper of Harvey--Lawson, \emph{Ann. Math.} \textbf{148} (1998), 1153--1154.

\bibitem{Machedon88}
Machedon, Matei, Szeg\"{o} kernels on pseudoconvex domains with one degenerate eigenvalue, \emph{Ann. of Math. (2)} \textbf{128} (1988), no. 3, 619–-640.

\bibitem{McNeal02}
McNeal, Jeffery D., A sufficient condition for compactness of the
{$\overline\partial$}-{N}eumann operator, \emph{J. Funct. Anal.}
\textbf{195} (2002), Nr. 1, 190--205.

\bibitem{MunasingheStraube07}
Munasinghe, Samangi and Straube, Emil J., Complex tangential flows and compactness of
the $\overline{\partial}$-Neumann operator, \emph{Pacific J. Math.} \textbf{232}, Nr.2
(2007),343--354.

\bibitem{MunasingheStraube12}
\bysame, Geometric sufficient conditions for compactness
of the complex Green operator, \emph{ J. Geom. Anal.} \textbf{22}, no.4 (2012),
1007--1026.

\bibitem{Nicoara06}
Nicoara, Andreea C., Global regularity for $\overline\partial\sb b$ on weakly
pseudoconvex CR manifolds, \emph{Adv. Math.} \textbf{199}, no.2  (2006), 356--447.

\bibitem{Raich10}
Raich, Andrew S., Compactness of the complex Green operator on CR-manifolds of
hypersurface type, \emph{Math. Ann.} \textbf{348} (2010), 81--117.

\bibitem{RaichStraube08}
Raich, Andrew S. and Straube, Emil J., Compactness of the complex Green operator,
\emph{Math. Res. Lett.} \textbf{15}, no.~4 (2008), 761--778.


\bibitem{Range02}
Range, R. Michael, Extension phenomena in multidimensional complex analysis: correction
of the historical record, \emph{Math. Intelligencer} \textbf{24}, no.2 (2002), 4--12.


\bibitem{Rothstein59}
Rothstein, Wolfgang, Bemerkungen zur Theorie komplexer R\"{a}ume, \emph{Math. Ann.} \textbf{137} (1959), 304–-315.

\bibitem{RothsteinSperling66}
Rothstein, W. and Sperling, H., Einsetzen analytischer {F}l\"achenst\"ucke in {Z}yklen auf komplexen {R}\"aumen, \emph{Festschr. Ged\"{a}chtnisfeier K. Weierstrass}, Westdeutscher Verlag, Cologne, 1966, 531--554. 

\bibitem{Shaw85}
Shaw, Mei-Chi, $L\sp 2$-estimates and existence theorems for the tangential
Cauchy--Riemann complex, \emph{Invent. Math.} \textbf{82}, no.1  (1985), 133--150.

\bibitem{Sibony87b}
Sibony, Nessim, Une classe de domaines pseudoconvexes, \emph{Duke
Math. J.} \textbf{55} , Nr. 2 (1987), 299--319.

\bibitem{Straube04}
Straube, Emil J., Geometric conditions which imply compactness of
the $\overline{\partial}$-Neumann operator, \emph{Ann. Inst. Fourier Grenoble} \textbf{54}, fasc. 3 (2004), 699--710.

\bibitem{Straube08}
\bysame, A sufficient condition for global regularity of the $\overline{\partial}$-Neumann operator, \emph{Adv. in Math.} \textbf{217} (2008), 1072--1095 .

\bibitem{Straube10a}
\bysame, \emph{Lectures on the $\mathcal{L}^{2}$-Sobolev Theory of the
$\overline{\partial}$-Neumann Problem}, ESI Lectures in Mathematics and Physics, European
Math. Society, Z\"{u}rich, 2010.

\bibitem{Straube10}
\bysame, The complex Green operator on CR-submanifolds of $\mathbb{C}^{n}$ of
hypersurface type: compactness, \emph{Trans. Amer. Math. Soc.} \textbf{364}, no.8 (2012),
4107--4125.

\bibitem{StraubeSucheston02}
Straube, Emil J. and Sucheston, Marcel K., Plurisubharmonic
defining functions, good vector fields, and exactness of a certain
one-form, \emph{Monatsh. f. Mathematik} \textbf{136} (2002), 249--258.

\bibitem{StraubeSucheston03}
\bysame, Levi foliations in pseudoconvex boundaries and vector fields that commute approximately with $\overline{\partial}$, \emph{Trans. Amer. Math. Soc.} \textbf{355}, Nr. 1 (2003), 143--154.


\bibitem{StraubeZeytuncu15}
Straube, Emil J. and Zeytun\c{c}u, Yunus E., Sobolev estimates for the complex Green operator on {CR} submanifolds of hypersurface type , \emph{Invent. Math.} \textbf{201}, (2015), 1073--1095.


\bibitem{Trepreau06}
Tr\'{e}preau, J.--M., Sur le prolongement holomorphe des fonctions {CR} d\'{e}finies sur une hypersurface r\'{e}elle de classe ${C}^2$ dans $\mathbb{C}^n$, \emph{Invent. math.} \textbf{83} (1986), 583--592.


\bibitem{Tumanov94}
Tumanov, A., Connections and propagation of analyticity for CR functions, \emph{Duke Math.
J.} \textbf{73}, no.1 (1994), 1--24.

\bibitem{Tumanov95}
\bysame, On the propagation of extendibility of CR functions, \emph{Complex analysis
and geometry} (Trento, 1993), 479--498, Lecture Notes in Pure and Appl. Math.,
\textbf{173}, Dekker, New York, 1996.

\bibitem{Tumanov96}
\bysame, Analytic discs and the extendibility of CR functions. \emph{Integral Geometry, Radon Transforms and Complex Analysis} (Venice, 1996), 123–141, Lecture Notes in Math., 1684, Springer, Berlin, 1998.

\bibitem{Zampieri08}
Zampieri, Giuseppe, \emph{Complex Analysis and CR-Geometry}, University Lecture Series
43, American Math. Soc., 2008.


\end{thebibliography}
\end{document}